\documentclass[11pt]{amsart}
\usepackage{amsmath}
\usepackage{amssymb}
\usepackage{amscd}
\usepackage{amsaddr}
\usepackage{amsfonts}
\usepackage{amsthm}
\usepackage{graphicx}
\usepackage{color}
\usepackage{enumerate}
\usepackage{tikz}
\usepackage{pgfplots}
\pgfplotsset{compat=1.14}
\usetikzlibrary{patterns,decorations.pathreplacing,calc,intersections,through}
\usepgfplotslibrary{fillbetween}
\usepackage[font=scriptsize]{caption}
\usepackage{verbatim}
\usepackage{float}
\usepackage{dsfont}
\usepackage[utf8]{inputenc}
\usepackage[T1]{fontenc}
\usepackage{bm}
\usepackage[outer=3cm, inner=3cm, top=1in, bottom=1.25in, marginparsep=0pt, marginparwidth=12pt]{geometry}
\newtheorem{theorem}{Theorem}[section]
\newtheorem{corollary}[theorem]{Corollary}
\newtheorem{fact}[theorem]{Fact}
\newtheorem{lemma}[theorem]{Lemma}
\theoremstyle{definition}
\newtheorem{definition}[theorem]{Definition}
\newtheorem{remark}[theorem]{Remark}
\newtheorem{example}[theorem]{Example}
\numberwithin{equation}{section}

\newcommand{\Me}[1][X]{\ensuremath{\mathcal{M}^{\mathsf{erg}}_{\sigma}\left(#1\right)}}
\newcommand{\dist}[2][x(0,...,n-1)]{\ensuremath{d^*\left(#1,#2 \right)}}
\newcommand{\U}{\ensuremath{\mathcal{U}}}
\newcommand{\V}{\ensuremath{\mathcal{V}}}
\newcommand{\M}{\ensuremath{\mathcal{M}}}

\newcommand{\T}{\ensuremath{\mathcal{T}}}

\newcommand{\N}{\ensuremath{\mathbb{N}}}
\newcommand{\bT}{\ensuremath{\bm{\mathsf{{T}}}}}

\newcommand*{\medcup}{\mathbin{\scalebox{1.5}{\ensuremath{\cup}}}}

\newcommand{\Z}{\ensuremath{\mathbb{Z}}}
\def \im {invariant measure}
\def \sq {sequence}
\title[Decomposition of a symbolic element into ergodic measures]{Decomposition of a symbolic element over a countable amenable group into blocks approximating ergodic measures}
\author{Tomasz Downarowicz}
\address{Faculty of Pure and Applied Mathematics, Wroc\l aw University
	of Science and Technology, Wybrze\.ze Wyspia\'nskiego 21, 50-370
	Wroc\l aw, Poland \\ (e-mail: Tomasz.Downarowicz@pwr.edu.pl)}
\author{Mateusz Wi\k{e}cek}
\address{Faculty of Pure and Applied Mathematics, Wroc\l aw University
	of Science and Technology, Wybrze\.ze Wyspia\'nskiego 21, 50-370
	Wroc\l aw, Poland\\ (e-mail: Mateusz.Wiecek@pwr.edu.pl)}
\begin{document}
	\thanks{The research of first author was supported by the NCN grant 2018/30/M/ST1/00061 and by the Wroc\l aw University of Science and Technology grant 8201003902.}
\maketitle
\begin{abstract}
	Consider a subshift over a finite alphabet, $X\subset \Lambda^\Z$ (or $X\subset\Lambda^{\N_0}$). With each finite block $B\in\Lambda^k$ appearing in $X$ we associate the \emph{empirical measure} ascribing to every block $C\in\Lambda^l$ the frequency of occurrences of $C$ in $B$. By comparing the values ascribed to blocks $C$ we define a metric on the combined space of blocks $B$ and probability measures $\mu$ on $X$, whose restriction to the space of measures is compatible with the weak-$\star$ topology. Next, in this combined metric space we fix an open set $\U$ containing all ergodic measures, and we say that a block $B$ is ``ergodic'' if $B\in\U$. 
	In this paper we prove the following main result: Given $\varepsilon>0$, every $x\in X$ decomposes as a concatenation of blocks of bounded lengths in such a way that, after ignoring a set $M$ of coordinates of upper Banach density smaller than $\varepsilon$, all blocks in the decomposition are ergodic. We also prove a finitistic version of this theorem (about decomposition of long blocks), and a version about decomposition of $x\in X$ into finite blocks of unbounded lengths. The second main result concerns subshifts whose set of ergodic measures is closed. We show that, in this case, no matter how $x\in X$ is partitioned into blocks (as long as their lengths are sufficiently large and bounded), after ignoring a set $M$ of upper Banach density smaller than $\varepsilon$, all blocks in the decomposition are ergodic. The first half of the paper is concluded by examples showing, among other things, that the small set $M$, in both main theorems, cannot be avoided.
	
	The second half of the paper is devoted to generalizing the two main results described above to subshifts $X\subset\Lambda^G$ with the action of a countable amenable group $G$. The role of long blocks is played by blocks whose domains are members of a F\o lner sequence while the decomposition of $x\in X$ into blocks (of which majority is ergodic) is obtained with the help of a congruent system of tilings.
\end{abstract}

\section*{Introduction}

In symbolic dynamics an invariant measure is determined by its values assumed on cylinders pertaining to finite blocks $C$. Given a long block $B$ we consider a function assigning to every block $C$ its \emph{frequency of occurrences} in $B$. In this manner, $B$ determines some kind of substitute of an invariant measure, which we call the \emph{empirical measure} associated to~$B$. Moreover, there is a natural metric measuring the distance between empirical measures associated to long blocks and invariant measures. Abusing slightly the terminology, we will say that the metric measures the distance between blocks and \im s. It is not hard to prove that any sufficiently long block $B$ occurring in a symbolic system $(X,\sigma)$ lies very close to some \im\ $\mu\in\M_\sigma(X)$, where $\M_\sigma(X)$ denotes the set of all shift-\im s supported by $X$.
 
On the other hand, it is a well-known fact that any \im\ $\mu\in\M_\sigma(X)$ decomposes as an integral average of ergodic measures supported by $X$. Henceforth, a question arises: Supposing that a long block $B$ appearing in $X$ is very close to an \im\ $\mu\in\M_\sigma(X)$, how is the ergodic decomposition of $\mu$ reflected in the structure of $B$?

Let us tentatively call a block $C$ \emph{ergodic} if it lies very close to some ergodic measure $\mu_C\in\M_\sigma^{\mathsf{erg}}(X)$ (by $\M_\sigma^{\mathsf{erg}}(X)$ we will denote the set of ergodic measures of $(X,\sigma)$). It is easy to see that if $B$ is a concatenation of ergodic blocks (not necessarily of equal lengths), say $B=C_1C_2,\dots,C_n$, then $B$ lies very close to the \im\ $\mu$ obtained as a convex combination (with appropriate coefficients) of the ergodic measures $\mu_{C_i}$, $i=1,2,\dots,n$. The question asked in the preceding paragraph takes on the following, more particular form: Is~being a concatenation of ergodic blocks \emph{the only possibility} for a long block $B$ to lie close to an \im? In this paper (among other things) we answer the above question positively after admitting a small correction in its formulation: 
\begin{enumerate}
	\item[($*$)] \emph{Every sufficiently long block $B$ appearing in a subshift $X$ decomposes as a concatenation of blocks of which vast majority (in terms of percentage of the total length) are ergodic.}
\end{enumerate} 
\noindent The above result is obtained as a corollary of a theorem stating that any sequence $x\in X$ can be decomposed into finite blocks such that the fraction of ergodic blocks (with respect to upper Banach density) is close to 1. The solution requires invoking subtle interplay between measures and blocks in symbolic systems as well as some properties of simplices in metric vector spaces.

 We comment that our result is interesting mainly for \emph{proper} subshifts. It is well known that in the set of \im s of the full shift ergodic measures lie densely. So, since any sufficiently long block $B$ lies very close to an \im, it lies equally close to an ergodic measure, i.e.\ $B$ is ergodic itself and needs not be decomposed any further.\footnote{However, even in case of the full shift our theorem does not completely trivialize. Since we can define ``ergodicity'' of blocks using an arbitrary open set around $\M_\sigma^{\mathsf{erg}}(X)$, not necessarily a ball with respect to  some distance, even for the full shift many long blocks can be classified as ``nonergodic''.} The problem becomes the less trivial the smaller (topologically) is the set of ergodic measures within $\M_\sigma(X)$. We pay a special attention to the case when $\M_\sigma(X)$ is a Bauer simplex, since then the ergodic measures form a closed, nowhere dense subset of $\M_\sigma(X)$.

 While the property $(*)$ of long blocks, may seem predictable for classical subshifts, an analogous property of ``blocks'', appearing in subshifts with an action of a general countable amenable group, is by far less obvious. It is a priori not even clear whether blocks with domains large enough to be close to invariant measures can be concatenated together. We made a (successful) attempt to overcome this and other difficulties and generalize our results to subshifts over countable amenable groups. 

 The paper is divided into five sections. Sections~\ref{sekcja1} and~\ref{sekcja4} are of preliminary character. The former pertains to classical symbolic systems with the action of $\Z$ or $\N_0$ (called also \emph{two-sided} and \emph{one-sided subshifts}, respectively), whereas the latter is concerned with subshifts over a general countable amenable group. Section~\ref{sekcja4} contains also an exposition on tilings and systems of dynamical tilings of amenable groups, which play an important role in section~\ref{sekcja5}. The first series of theorems concerning the decomposition of a symbolic element of (as well as a sufficiently long block appearing in) a classical subshift $X$ into blocks approximating ergodic measures is formulated and proved in section~\ref{sekcja2}. It is shown that for any open neighbourhood $\mathcal{U}$ of the set of ergodic measures of a symbolic system $X$ and any positive $\varepsilon$, for every $x\in X$, \textbf{there exists} a decomposition of $x$ into finite blocks of bounded lengths, such that the domains of those blocks which do not lie in $\U$ cover a set in $\Z$ (or $\N_0$) of upper Banach density smaller than $\varepsilon$. A small modification of the proof allows us to deduce that for every $x\in X$ there exists also a decomposition into finite blocks of \emph{unbounded} lengths, such that the domains of blocks not lying in $\U$ cover a set of upper Banach density~$0$. Moreover, it is proved that in a subshift $X$ for which shift-invariant measures form a Bauer simplex, \textbf{for any} decomposition of an element $x\in X$ into sufficiently long blocks, the fraction (with respect to  upper Banach density) of those blocks which do not lie in $\U$ is smaller than $\varepsilon$. In section~\ref{sekcja3} we provide three examples showing that the assumptions in theorems from section~\ref{sekcja2} cannot be omitted. Section~\ref{sekcja5} is dedicated to generalizing the main results of section~\ref{sekcja2} to the case of symbolic systems with an action of a countable amenable group $G$. In this case, the role of long blocks is played by ``blocks'', whose domains are sets with good F{\o}lner properties. Our methodology heavily relies on the theory of tilings and congruent systems of dynamical tilings, explained in section~\ref{sekcja3}.
\section{Classical symbolic systems}\label{sekcja1}
All theorems provided in this section are standard and their proofs are omitted.

 Let $\Lambda$ be a finite, discrete space called an \emph{alphabet}. By a classical \emph{symbolic system} with the action of $\Z$ (resp.\ $\N_0$) we mean a two-sided (resp.\ one-sided) subshift, i.e. any subset $X$ of $\Lambda^{\Z}$ (resp.\ $\Lambda^{\N_0}$), which is closed and invariant under the \emph{shift} transformation $\sigma$ given by
\begin{equation*}
(\sigma(x))(i)=x(i+1),\ \ \ x\in \Lambda^{\Z},\ i\in\Z\ (\text{resp.\ }x\in\Lambda^{\N_0},\ i\in\N_0).
\end{equation*}

 From now on, to avoid repeating that $i$ ranges over either $\Z$ or $\N_0$, depending on the type of subshift, we will skip indicating the range of that index. By a \emph{block} of length $k$ we mean any element $B=(B(0),B(1),\dots, B(k-1))\in\Lambda^k$. The length $k$ of the block $B$ is also denoted by $|B|$. If, for some $x\in X$ and $i$, we have $\sigma^i(x)|_{[0,k)}=B$, we say that the block $B$ \emph{occurs} in $x$ at the position $i$. Abusing slightly the notation, we write $x|_{[i,i+k)}=B$ and call the interval $[i,i+k)$ \emph{the domain of the occurrence of the block $B$ in $x$}. A similar convention is applied to ``subblocks'' of blocks: If $B\in\Lambda^k$  and $[n,m)\subset [0,k)$, then by $B|_{[n,m)}$ we mean the block $C\in \Lambda^{m-n}$ defined by $C(l)=B(n+l)$ for all $l=0,\dots,m-n-1$. The set of all blocks occurring in the elements of $X$ is denoted by $\mathcal{B}^*(X)$.
\begin{definition} Let $B\in\Lambda^{k}$ and $C\in\Lambda^l$, where $l\le k$. The \emph{frequency of occurrences of the block $C$ in the block $B$} is defined as
	\begin{equation}
	\mathrm{Fr}_B(C)=\frac{|\{i\in[n,n+k-l]: B|_{[i,i+l)}= C\}|}{k}.
	\end{equation}
	If $|C|>|B|$, we let $\mathrm{Fr}_B(C)=0$.
\end{definition}

 For a fixed block $B\in\Lambda^k$, frequencies of occurrences of blocks $C\in\Lambda^l$ in $B$, where $l\le k$, form a sub-probability vector. Using the notion of frequency of occurrences of one block in another, we define a distance between two blocks $B_1,B_2\in\mathcal{B}^*(X)$ by
\begin{equation}\label{dist_blocks}
d^*(B_1,B_2)=\sum_{l=1}^{+\infty}2^{-l}\sum_{C\in\Lambda^l}|\mathrm{Fr}_{B_1}(C)-\mathrm{Fr}_{B_2}(C)|.
\end{equation}

 In what follows, $\mathcal{M}(X)$ denotes the set of all Borel probability measures on $X$, while $\mathcal{M}_{\sigma}(X)\subset\mathcal{M}(X)$ denotes the set of all shift-invariant measures (i.e.\ such that $\mu(A)=\mu(\sigma^{-1}(A))$ for every Borel set $A\subset X$). Further, $\Me[X]\subset\mathcal{M}_{\sigma}(X)$ denotes the set of all ergodic measures (i.e.\ such that $\mu(A\,\triangle\,\sigma^{-1}(A))=0\Rightarrow \mu(A)\in\{0,1\}$, for any Borel set $A\subset X$). Note that $\mathcal{M}_{\sigma}(X)$ is a closed, convex subset of $\mathcal{M}(X)$, which is compact in weak-$\star$ topology. It is well known that the extreme points of $\mathcal{M}_{\sigma}(X)$ are the ergodic measures.

 The formula~\eqref{dist_blocks} is similar to the one defining the standard metric on~$\mathcal{M}(X)$:
\begin{equation}\label{dist_measures}
d^*(\mu_1,\mu_2)=\sum_{l=1}^{+\infty}2^{-l}\sum_{C\in\Lambda^l}|\mu_1([C])-\mu_2([C])|, \ \mu_1,\mu_2\in \mathcal{M}(X),
\end{equation}
where $[C]=\{x\in X: x|_{[0,l)}=C\}$ denotes the \emph{cylinder associated with the block $C$}. Note that the above metric is compatible with the weak-$\star$ topology on $\mathcal{M}(X)$. Henceforth, we can define the distance between a block and an invariant measure~by
\begin{equation}\label{dist_block_measure}
d^*(B,\mu)=\sum_{l=1}^{+\infty}2^{-l}\sum_{C\in\Lambda^l}|\mathrm{Fr}_B(C)-\mu([C])|, \ B\in\mathcal{B}^*(X),\ \mu\in \mathcal{M}(X).
\end{equation}
Then, the function $d^*$, given by the equations~\eqref{dist_blocks},~\eqref{dist_measures} and~\eqref{dist_block_measure} is a metric on the set $\mathcal{B}^*(X)\cup\mathcal{M}(X)$. Moreover, it turns out that sufficiently long blocks lie uniformly close to the set $\mathcal{M}_{\sigma}(X)$, as stated in the next theorem (see~ \cite[Fact 6.6.1]{Entropy}).
\begin{theorem}\label{blocks_close_to_measures}
	Fix an $\varepsilon>0$. There exists $l_0\in\N$ such that for all $l\ge l_0$ and every block $B\in\mathcal{B}^*(X)$ of length $l$ we have
	\begin{equation*}
	d^*(B,\mathcal{M}_{\sigma}(X))<\varepsilon.
	\end{equation*}
\end{theorem}

 The connection between blocks and measures enables us to give an (alternative to the general one involving the measures $\frac1n\sum_{i=0}^{n-1}\delta_{\sigma^{i}(x)}$) definition of a generic point, in a symbolic system, for a $\sigma$-invariant measure using the metric $d^*$.
\begin{definition}\label{generic}
	An element $x\in X$ is called \emph{generic} for a measure $\mu\in\mathcal{M}_{\sigma}(X)$ if $$\lim_{n\to+\infty}d^*(x|_{[0,n)}, \mu)=0.$$
\end{definition}

 A lemma similar to the one below can be found (in a more general version) for example in \cite[lemma 2]{Downarowicz}.
\begin{lemma}\label{konkatenacja}
	For every $\varepsilon>0$ and a finite collection of blocks $B_1,...,B_m\in\mathcal{B}^*(X)$, such that for every $j=1,...,m$ there exists a measure $\mu_j\in\mathcal{M}(X)$ satisfying $d^*(B_j,\mu_j)<\varepsilon$, the following inequality holds
	\begin{equation*}
	d^*(B,\mu)<2\varepsilon,
	\end{equation*}
	where $B=B_1...B_m$ is the concatenation of the blocks $B_1,...,B_m$ and $\mu=\sum_{j=1}^{m}\frac{|B_j|}{|B|}\mu_j$.
\end{lemma}

 We finish this section by providing the definitions of upper and lower Banach densities of subsets of $\Z$ or $\N_0$,  and some of their properties.
\begin{definition}\label{densities}
	Fix a set $A\subset \Z$ (resp.\ $A\subset\N_0$). \emph{Upper Banach density} of $A$ is defined as
	\begin{equation}\label{upper_banach_density}
	\overline{d}_{\mathsf{Ban}}(A)=\inf_{N\in\N}\sup_i\frac{|A\cap[i,i+N)|}{N}.
	\end{equation}
	Similarly, we define \emph{lower Banach density} as
	\begin{equation}\label{lower_banach_density}
	\underline{d}_{\mathsf{Ban}}(A)=\sup_{N\in\N}\inf_i\frac{|A\cap[i,i+N)|}{N}.
	\end{equation}
	 If $\overline{d}_{\mathsf{Ban}}(A)=\underline{d}_{\mathsf{Ban}}(A)$, then we denote the common value by $d_{\mathsf{Ban}}(A)$ and call it \emph{Banach density} of $A$.
\end{definition}
\begin{remark}\label{lower_upper}
	It follows directly from the definition of upper and lower Banach densities that for every set $A\subset \Z$ (resp.\ $A\subset \N_0$) we have $\overline{d}_{\mathsf{Ban}}(A)=1-\underline{d}_{\mathsf{Ban}}(A^{\mathsf{c}})$.
\end{remark}
\begin{theorem}
	The upper Banach density is subadditive, that is for every pair of sets $A,B\subset \Z$ (resp.\ $A,B\subset \N_0$) the following inequality holds
	\begin{equation*}
	\overline{d}_{\mathsf{Ban}}(A\cup B)\le \overline{d}_{\mathsf{Ban}}(A)+\overline{d}_{\mathsf{Ban}}(B).
	\end{equation*}
\end{theorem}

 The next lemma, connecting the notions of upper Banach density and invariant measures, is true not only for symbolic systems but for every topological dynamical system with an action of $\Z$ or $\N_0$ on a compact, metric space. Hence we formulate it in this general setup. 
\begin{lemma}\label{upper_banach_density_invariant}
	Let $(X,T)$ be a topological dynamical system with an action of $\Z$ or $\N_0$ consisting of the iterates of a homeomorphism (resp.\ continuous map) $T:X\to X$, where $X$ is a compact metric space, and let $D\subset X$ be a closed set. The following inequality holds
	\begin{equation*}
	\sup_{\mu\in\mathcal{M}_T(X)}\mu(D)\ge \sup_{x\in X}\ \overline{d}_{\mathsf{Ban}}(\{i: T^i(x)\in D\}),
	\end{equation*}
	where $\mathcal{M}_T(X)$ denotes the set of all $T$-invariant, Borel probability measures on $X$.
\end{lemma}
\section{Decomposition of an element of a classical subshift into ergodic blocks}\label{sekcja2}

 We begin with a rigorous definition of the ``ergodic blocks'' alluded to in the introduction.
\begin{definition}
	Let $\U$ be an open set in $\mathcal{B}^*(X)\cup\mathcal{M}(X)$ such that $\U\supset \Me[X]$. A block $B$ is called \emph{$\U$-ergodic} if $B\in\U$. All other blocks are shortly called \emph{nonergodic}.
\end{definition}
\begin{lemma}\label{Dum}
	Let $(X,\sigma)$ be a classical symbolic system and let $\U\supset \Me[X]$ be an open set in $\mathcal{B}^*(X)\cup\mathcal{M}(X)$. For every $m\in\N$, the set
	\begin{equation*}
	D_{\U,m}=\{y\in X: \forall_{k\ge m}\ y|_{[0,k)}\notin \U\}
	\end{equation*}
	is a null-set for every measure $\mu\in\mathcal{M}_{\sigma}(X)$.
\end{lemma}
\begin{proof}
	If for some measure $\mu\in\mathcal{M}_{\sigma}(X)$ we had $\mu(D_{\U,m})>0$, then by the ergodic decomposition, there would exist an ergodic measure $\mu_0\in\Me[X]$ such that ${\mu_0(D_{\U,m})>0}$. Hence, by the Birkhoff Ergodic Theorem, in $D_{\U,m}$ there would exist an element $x$, generic for the measure $\mu_0$. Therefore, for some $k\ge m$ we would have $x|_{[0,k)}\in \U$. This would contradict the definition of the set $D_{\U,m}$.
\end{proof}
\begin{lemma}\label{Eumn}
	Let $(X,\sigma)$ be a classical symbolic system, let $\U\supset \Me[X]$ be an open set in $\mathcal{B}^*(X)\cup\mathcal{M}(X)$ and let $m\in\N$. For every $n\ge m$ we define
	\begin{equation*}
	D_{\U,m,n}=\{y\in X: \forall_{k\in [m,n]}\ y|_{[0,k)}\notin \U\}
	\end{equation*}
	and for all $x\in X$ we define
	\begin{equation*}
	E_{\U,m,n,x}=\{i: \sigma^i(x)\in D_{\U,m,n}\}.
	\end{equation*}
	Then the convergence $\lim\limits_{n\to+\infty}\overline{d}_{\mathsf{Ban}}(E_{\U,m,n,x})=0$ holds uniformly on $X$.
\end{lemma}
\begin{proof}
	The sequence of sets $(D_{\U,m,n})_{n\ge m}$ is obviously nested, hence for every $x\in X$, the sequence $(E_{\U,m,n,x})_{n\ge m}$ is also nested. Thus, the sequence $(\overline{d}_{\mathsf{Ban}}(E_{\U,m,n,x}))_{n\ge m}$ is nonincreasing.
	
	 The set $D_{\U,m,n}=\bigcap_{k=m}^n\{y\in X: y|_{[0,k)}\notin\U\}$ is clopen for every $n\ge m$, thence the function $\Phi_{\U,m,n}(\mu)=\mu(D_{\U,m,n})$ is continuous in the weak-$\star$ topology on $\mathcal{M}(X)$. Moreover, the set $D_{\U,m}$ is the intersection of the sets $D_{\U,m,n}$, $n\ge m$. By lemma~\ref{Dum} and the continuity of measures from above, it follows that $\Phi_{\U,m,n}\to 0$ as $n\to+\infty$, pointwise on $\mathcal{M}_{\sigma}(X)$. Because the sequence of functions $(\Phi_{\U,m,n})_{n\ge m}$ is nonincreasing, by Dini's theorem it tends to~$0$ uniformly on $\mathcal{M}_{\sigma}(X)$. Henceforth $\displaystyle\sup_{\mu\in\mathcal{M}_{\sigma}(X)}\mu(D_{\U,m,n})$ tends to $0$ as $n\to+\infty$. By lemma~\ref{upper_banach_density_invariant}, $\displaystyle\sup_{x\in X}\overline{d}_{\mathsf{Ban}}(E_{\U,m,n,x})\to 0$, hence the functions $x\mapsto \overline{d}_{\mathsf{Ban}}(E_{\U,m,n,x})$ converge to $0$ as~$n\to+\infty$, uniformly on $X$, as desired.
\end{proof}

 Now we formulate and prove one of the key theorems of this paper, concerning a decomposition of a symbolic element into $\U$-ergodic blocks in the case of classical symbolic systems. In section~\ref{sekcja5} this theorem will be generalized to the case of a symbolic system with an action of any countable amenable group.
\begin{theorem}\label{thm1}
	Let $(X,\sigma)$ be a classical symbolic system. Let $\U\supset \Me[X]$ be an open set in $\mathcal{B}^*(X)\cup\mathcal{M}(X)$. For every $m\in\N$ and every $\varepsilon>0$ there exists $n\ge m$ such that, for any element $x\in X$, there exists a representation of $x$ as an infinite concatenation of blocks: $x=\dots B_{-2}A_{-2}B_{-1}A_{-1}A_1B_1A_2B_2\dots$ (resp.\ $x=A_1B_1A_2B_2\ldots$, in the case of a one-sided subshift), where all blocks $B_j$ are $\U$-ergodic and have lengths ranging between $m$ and $n$, and the set $M^{\mathsf{NE}}(x)$ of coordinates pertaining to the blocks $A_j$ in this concatenation has upper Banach density smaller than $\varepsilon$ (some of the blocks $A_j$ may be empty, i.e.\ have length zero).
\end{theorem}
\begin{proof}
	By lemma~\ref{Eumn}, there exists $n\ge m$ such that $\sup_{x\in X}\overline{d}_{\mathsf{Ban}}(E_{\U,m,n,x})<\varepsilon$. Hence, for every $x\in X$, the set $(E_{\U,m,n,x})^{\mathsf{c}}=\{i: \sigma^i(x)\notin D_{\U,m,n}\}$ has positive lower Banach density (in particular it is infinite). The desired decomposition of an element $x\in X$ can be obtained as~follows:
	
	 We find the smallest $i\ge 0$ such that $\sigma^i(x)\notin D_{\U,m,n}$ and denote it by $i_1$. We define $A_1=x|_{[0,i_1)}$. There exists $n_1\in[m,n]$ such that $\sigma^{i_1}(x)|_{[0,n_1)}\in \U$. We define $B_1$ as the block $x|_{[i_1,i_1+n_1)}$. Next, we find the smallest $i\ge i_1+n_1$ such that $\sigma^i(x)\notin D_{\U,m,n}$ and denote it by~$i_2$. We define $A_2=x|_{[i_1+n_1,i_2)}$. As before, there exists $n_2\in[m,n]$ satisfying $\sigma^{i_2}(x)|_{[0,n_2)}\in \U$. We define $B_2=x|_{[i_2,i_2+n_2)}$. In the same way we define the blocks $A_j$ and $B_j$ for $j\ge 3$. In case of the action of $\Z$, the blocks with negative indices~$j$ are defined analogously, proceeding to the left from the coordinate 0, using the fact that $\M_{\sigma^{-1}}(X)=\M_{\sigma}(X)$.
	
	 For every index $j$ we have $B_j\in\U$ and $m\le n_j\le n$. Furthermore, $$M^{\mathsf{NE}}(x)=[0,i_1)\cup\bigcup_{j}[i_j+n_j,i_{j+1})$$ is contained in $E_{\U,m,n,x}$. Henceforth, by lemma~\ref{Eumn}, this set has upper Banach density smaller than $\varepsilon$.
\end{proof}

 The above theorem allows one to deduce a finitistic result announced in the introduction.
\begin{theorem}\label{finitistic}
	Let $(X,\sigma)$ be a classical symbolic system and let $\U\supset \Me[X]$ be an open set in $\mathcal{B}^*(X)\cup\M(X)$. For every $\varepsilon>0$ there exists $n_0\in\N$ such that every block $B\in\mathcal{B}^*(X)$ of length at least $n_0$ may be decomposed as a concatenation of blocks: $B=A_1B_1A_2B_2\dots A_rB_rA_{r+1}$, such that the blocks $B_1,\dots,B_r$ are $\U$-ergodic and $\sum_{j=1}^{r}|B_j|\ge (1-\varepsilon)|B|$.
\end{theorem}
\begin{proof}
	It suffices to consider a one-sided subshift $X$. By contradiction, suppose there exist $\varepsilon>0$ and a sequence of blocks $(C_s)_{s\in\N}$ such that the lengths $|C_s|$ increase to $+\infty$, and for each $s\in\N$ the block $C_s$ cannot be decomposed as described in the formulation of the theorem. Let $x$ be the infinite concatenation $x=C_1C_2C_3\dots$. Although $x$ need not belong to $X$, the symbolic system $X'=X\cup\overline{O_{\sigma}(x)}$, where $\overline{O_{\sigma}(x)}$ denotes the orbit-closure of the element $x$, has the same collection of invariant measures as $X$. Fix $m\in\Z$. By theorem~\ref{thm1} (applied to $X'$), there exists $n\ge m$ such that $x$ may be decomposed as the infinite concatenation of blocks $x=A_1B_1A_2B_2\dots$, where all blocks $B_j$ are $\U$-ergodic and have lengths ranging between $m$ and $n$, and the set $M^{\mathsf{NE}}(x)$ of coordinates pertaining to the blocks $A_j$ has upper Banach density smaller than $\varepsilon$. By the definition of upper Banach density, there exists $N_0\in\N$ such that for every $N\ge N_0$ and every $i$ we have
	\begin{equation*}
	\frac{|M^{\mathsf{NE}}(x)\cap [i,i+N)|}{N}<\frac{\varepsilon}2.
	\end{equation*}
	
	 For $s$ sufficiently large, we have $|C_s|\ge \max\{N_0,\tfrac{4n}{\varepsilon}\}$. Let $t_s$ be such that ${x|_{[t_s,t_s+|C_s|)}=C_s}$. Let $l$ and $r$ be such that $B_l$ and $B_{l+r}$ are the first and the last of the $\U$-ergodic blocks in the concatenation $A_1B_1A_2B_2\dots$ representing $x$, entirely covered by the block $C_s$. It is now elementary to see that
	\begin{equation}\label{impossible}
	C_{s}=\tilde{A}_lB_{l}A_{l+1}B_{l+1}\dots A_{l+r}B_{l+r}\tilde{A}_{l+r+1},
	\end{equation}
	where $\tilde{A}_l$ is a subblock of the block $B_{l-1}A_l$ and $\tilde{A}_{l+r+1}$ is a subblock of $A_{l+r+1}B_{l+r+1}$. Recall that $|B_{l-1}|+|B_{l+r+1}|\le 2n$. Hence, the $\U$-ergodic subblocks of $C_s$, that is $B_l,\dots,B_{l+r}$, satisfy the inequality
	\begin{equation*}
	\frac{1}{|C_s|}\sum_{j=l}^{l+r}|B_j|\ge 1-\frac{|M^{\mathsf{NE}}(x)\cap [t_s,t_s+|C_s|)|}{|C_s|}-\frac{2n}{|C_s|}>1-\varepsilon.
	\end{equation*}
	Finally note that the blocks $B_l,\dots,B_{l+r}$ appear not only in $X'$ but, as subblocks 
	of $C_s$, also in $X$. Thus the formula~\ref{impossible} gives a decomposition of $C_s$ in a way assumed to be impossible. This contradiction ends the proof.
\end{proof}
\begin{remark}
	Theorem~\ref{finitistic} implies that in theorem~\ref{thm1} there exists an independent of $x\in X$ ``horizon'' $N_0$ allowing to determine that the upper Banach density of $M^{\mathsf{NE}}(x)$ is smaller than~$\varepsilon$.
\end{remark}
\begin{theorem}\label{unbounded}
		Let $(X,\sigma)$ be a classical symbolic system. Let $\U\supset \Me[X]$ be an open set in $\mathcal{B}^*(X)\cup\mathcal{M}(X)$. For every $m\in\N$ and $x\in X$ there exists a representation of $x$ as an infinite concatenation of finite blocks: $x=\dots B_{-2}A_{-2}B_{-1}A_{-1}A_1B_1A_2B_2\dots$ (resp.\ $x=A_1B_1A_2B_2\ldots$, in the case of a one-sided subshift), where all blocks $B_j$ are $\U$-ergodic and have lengths larger than or equal to $m$ (without an upper bound on the lengths) and and the set $M^{\mathsf{NE}}(x)$ of coordinates pertaining to the blocks $A_j$ in this concatenation has Banach density equal to $0$.
\end{theorem}
\begin{proof}
	We define the set
	$D_{\U,m}=\{y\in X: \forall_{k\ge m}\ y|_{[0,k)}\notin \U\}$ and, for every point $x\in X$, we define the set $E_{\U,m,x}=\{i: \sigma^i(x)\in D_{\U,m}\}$.
	Observe that $D_{\U,m}=\bigcap_{m\ge n}D_{\U,m,n}$. Hence, for every $x\in X$ we have $E_{\U,m,x}=\bigcap_{n\ge m}E_{\U,m,n,x}$. So $\overline{d}_{\mathsf{Ban}}(E_{\U,m})\le \inf_{n\ge m}\overline{d}_{\mathsf{Ban}}(E_{\U,m,n})=0$.
	The construction of the required decomposition of an element $x\in X$ is a straightforward modification of that in the proof of theorem~\ref{thm1}, consisting in replacing the set $E_{\U,m,n,x}$ by~$E_{\U,m,x}$.
\end{proof}

 In the special case, when $\M_{\sigma}(X)$ is a Bauer simplex, that is, the set of ergodic measures $\Me[X]$ is closed in $\mathcal{M}(X)$, we will prove a stronger version of theorem~\ref{thm1}. The strengthening consists in replacing the phrase ``\textbf{there exists} a representation of $x$ as an infinite concatenation of finite blocks'' with ``\textbf{for any} representation of $x$ as an infinite concatenation of sufficiently long blocks''. In the proof we use the following two lemmas concerning compact, convex sets in locally-convex, metric vector spaces. Rather standard proofs of these lemmas are omitted.
\begin{lemma}\label{barycenter_mass}
	Let $\M$ be a compact, convex subset of a locally-convex, metric vector space $\mathbb{V}$ and let $d^*$ denote a convex metric on $\mathbb{V}$. Let $\mu_0$ be an extreme point of $\M$. For every $\varepsilon>0$ there exists $\gamma>0$ such that for every $\mu=\int_{\M}\nu\mathrm{d}\xi(\nu)$ for some Borel probability measure $\xi$ on $\M$ (that is $\mu$ is a so-called \emph{barycenter} of the measure $\xi$), the following implication holds 
	\begin{equation*}
	d^*(\mu,\mu_0)<\gamma \Rightarrow \xi\bigl(\M\setminus \mathsf{Ball}(\mu_0,\varepsilon)\bigr)<\varepsilon.
	\end{equation*}
\end{lemma}

 When the set of extreme points of $\M$ is closed, lemma~\ref{barycenter_mass} may be strengthened as follows.
\begin{lemma}\label{barycenter_compact}
	Let $\M$ be a compact, convex subset of a locally-convex, metric vector space $\mathbb{V}$ and let $d^*$ denotes a convex metric on $\mathbb{V}$. Assume that the set $\mathrm{ex}(\M)$ of extreme points of $\M$ is closed. Then, for every $\varepsilon>0$, there exists $\gamma>0$ such that for every pair $\mu,\mu_0$, where $\mu_0\in\mathrm{ex}(\M)$ and $\mu=\int_{\M}\nu\mathrm{d}\xi(\nu)$ for some Borel probability measure $\xi$ on $\M$, the following implication is true
	\begin{equation*}
	d^*(\mu,\mu_0)<\gamma\Rightarrow\xi\bigl(\M\setminus \mathsf{Ball}(\mu_0,\varepsilon)\bigr)<\varepsilon.
	\end{equation*}
\end{lemma}
\begin{theorem}\label{thm2}
	Let $(X,\sigma)$ be a classical symbolic system such that $\M_{\sigma}(X)$ is a Bauer simplex. Let $\U\supset \Me[X]$ be an open set in $\mathcal{B}^*(X)\cup\mathcal{M}(X)$. Then, for every $\varepsilon>0$, there exists $k_0\in\N$ such that for every $x\in X$ and any decomposition $x=\dots C_{-2}C_{-1}C_1C_2\dots$ (resp.\ $x=C_1C_2\dots$, in the case of a one-sided subshift) into blocks $C_j$ of lengths larger than or equal to $k_0$ and bounded from above, the set of coordinates pertaining to nonergodic (i.e.\ which are not $\U$-ergodic) blocks $C_j$ in the above concatenation, has upper Banach density smaller than $\varepsilon$.
\end{theorem}

\begin{remark}It is worth mentioning that theorem~\ref{thm2} applies to any partition of $x$ into blocks of equal (sufficiently large) lengths.
\end{remark}

\begin{proof}[Proof of theorem~\ref{thm2}]
	Choose an $\varepsilon>0$. By compactness of $\Me[X]$, without loss of generality, we can assume that $\U=\mathsf{Ball}(\Me[X],\rho)$ for some $\rho>0$. Then we can also assume that $\varepsilon=\rho$.
	
	 Since $\M_{\sigma}^{\mathsf{erg}}(X)$ is closed, lemma~\ref{barycenter_compact} implies that there exists $0<\gamma<\varepsilon$ such that for each pair of measures $\mu_0\in\M_{\sigma}^{\mathsf{erg}}(X)$ and $\mu=\int_{\mathcal{M}_{\sigma}(X)}\nu\mathrm{d}\xi(\nu)$, the following implication holds
	\begin{equation*}
	d^*(\mu,\mu_0)<\gamma \Rightarrow \xi\bigl(\M_{\sigma}(X)\setminus \mathsf{Ball}(\mu_0,\tfrac{\varepsilon}2)\bigr)<\tfrac{\varepsilon}2.
	\end{equation*}
	
	 By theorem~\ref{blocks_close_to_measures}, there exists $k_0$ such that every block $C\in\mathcal{B}^*(X)$ of length at least $k_0$ satisfies $\dist[C]{\mathcal{M}_{\sigma}(X)}<\tfrac{\gamma}4$. Choose some $k\ge k_0$ and an $x\in X$. Fix some decomposition of $x$ into blocks: $x=\dots C_{-2}C_{-1}C_1C_2\dots$ (resp.\ $x=C_1C_2C_3...$, for a one-sided subshift), of lengths from the interval $[k_0,k]$. Let $m_0=\lceil \tfrac{8k}{\gamma}\rceil$.
	
	 By theorem~\ref{thm1}, there exist $n_0\in\mathbb{N}$ and a representation: $x=\dots B_{-1}A_{-1}A_1B_1\dots$ (resp.\ $x=A_1B_1A_2B_2\dots$, in the case of a one-sided subshift), such that the blocks $B_i$ have lengths from the interval $[m_0,n_0]$ and satisfy $\dist[B_i]{\Me[X]}<\tfrac{\gamma}4$, and  $\overline{d}_{\mathsf{Ban}}\bigl(M^{\mathsf{NE}}(x)\bigr)<\tfrac{\varepsilon}4$. Let us denote by $I_{A_i}$ (resp.\ $I_{B_i}$, $I_{{C_j}}$) the sets of coordinates pertaining to the block $A_i$ (resp.~$B_i$,~$C_j$). We define the sets
	\begin{gather*}
	\mathcal{J}_{B_i}=\{j: I_{C_j}\subset I_{B_i}\}, \hspace{20pt} \mathcal{J}_{\mathsf{B}}=\bigcup_{i}\mathcal{J}_{B_i}\\
	\mathcal{J}_{A_i}=\{j: I_{C_j}\cap I_{A_i}\neq\varnothing \}, \hspace{20pt} \mathcal{J}_{\mathsf{A}}=\bigcup_{i}\mathcal{J}_{A_i}.
	\end{gather*}
	Furthermore, for each $i$ we define $\tilde{B_i}$ as the concatenation of the blocks $C_j$ with $j\in\mathcal{J}_{B_i}$. The construction of the blocks $\tilde{B}_i$ is presented in figure 1.
	\begin{figure}[ht]
		\centering
		\begin{tikzpicture}
		\draw (-7,0.5)--(7.2,0.5);
		\draw (-7,-0.5)--(7.2,-0.5);
		\draw (-7,0.3)--(7.2,0.3);
		\draw (-7,-0.3)--(7.2,-0.3);
		\foreach \x in {-14,...,14}
		\draw (0.5*\x,0.5)--(0.5*\x, 0.3);
		\foreach \x in {-7,-5.75,0.25,1.35,6.1,6.8}
		\draw (\x,-0.5)--(\x,-0.3);
		\foreach \x in {1,...,28}
		\node[anchor= north west] at ({-7.5+0.5*\x}, 0.3) {\tiny$C_{\x}$};
		\node[anchor= north] at (-6.25, -0.5) {$A_1$};
		\node[anchor= north] at (-2.75, -0.5) {$B_1$};
		\node[anchor= north] at (0.9, -0.5) {$A_2$};
		\node[anchor= north] at (3.75, -0.5) {$B_2$};
		\node[anchor= north] at (6.5, -0.5) {$A_3$};
		\draw [
		thick,
		decoration={
			brace
		},
		decorate
		] (-5.5,0.6) -- (0,0.6);
		\draw [
		thick,
		decoration={
			brace
		},
		decorate
		] (1.5,0.6) -- (6,0.6);
		\node[anchor=south] at (-2.75,0.6) {$\tilde{B_1}$};
		\node[anchor=south] at (3.75,0.6) {$\tilde{B_2}$};
		\end{tikzpicture}
		\caption{The construction of the blocks $\tilde{B_i}$}
		\label{blocks_scheme}
	\end{figure}
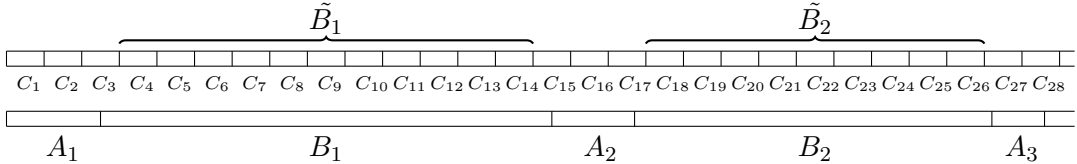

	\noindent Note that, for each $i$, we have $$\frac{|\tilde{B}_i|}{|B_i|}\ge 1-\frac{2k}{m_0}>1-\frac{\gamma}{4},$$
	which implies that the lower Banach density of the coordinates pertaining to the blocks $\tilde{B}_i$ is larger than or equal to the lower Banach density of the coordinates pertaining to the blocks $B_i$ multiplied by $1-\tfrac{\gamma}{4}$. Passing to the complements, we get \begin{equation}\label{dbA}
	\overline{d}_{\mathsf{Ban}}\bigl(\bigcup_{j\in\mathcal{J}_{\mathsf{A}}}I_{C_j}\bigr)\le \overline{d}_{\mathsf{Ban}}\bigl(M^{\mathsf{NE}}(x)\bigl)+\frac{\gamma}{4}<\frac{\varepsilon}2.
	\end{equation}
	For every $i$ we have
	\begin{equation*}
	d^*(B_i,\tilde{B_i})=\sum_{l\in\mathbb{N}}\frac{1}{2^{l}}\sum_{D\in\Lambda^l}|\mathrm{Fr}_{B_i}(D)-\mathrm{Fr}_{\tilde{B_i}}(D)|\le \frac{2k}{m_0}<\frac{\gamma}{4}.
	\end{equation*}
	Additionally, for every $i$, we have $\dist[B_i]{\Me[X]}<\tfrac{\gamma}4$. Hence, there exists a measure $\mu_{i}\in\M_{\sigma}^{\mathsf{erg}}(X)$ satisfying $\dist[\mu_{i}]{B_i}<\tfrac{\gamma}4$. On the other hand, for every $j$, there exists a measure $\nu_j\in\mathcal{M}_{\sigma}(X)$ such that $d^*(C_j,\nu_j)<\tfrac{\gamma}4$.
	By lemma~\ref{konkatenacja}, for every $i$, it is true that
	\begin{equation*}
	\dist[\tilde{B_i}]{\sum_{j\in \mathcal{J}_{B_i}}\frac{|C_j|}{|\tilde{B_i}|}\nu_j}<\frac{\gamma}{2}.
	\end{equation*}
	Thus, using the triangle inequality, we obtain
	\begin{equation*}
	\dist[\mu_{i}]{\sum_{j\in \mathcal{J}_{B_i}}\frac{|C_j|}{|\tilde{B_i}|}\nu_j}\le \dist[\mu_{i}]{B_i}+\dist[B_i]{\tilde{B_i}}+\dist[\tilde{B_i}]{\sum_{j\in \mathcal{J}_{B_i}}\frac{|C_j|}{|\tilde{B_i}|}\nu_j}<\gamma.
	\end{equation*}
Obviously, $\sum_{j\in \mathcal{J}_{B_i}}\frac{|C_j|}{|\tilde{B_i}|}\nu_j=\int_{\mathcal{M}_{\sigma}(X)}\nu\mathrm{d}\xi_i(\nu)$ for the measure $\xi_i=\sum_{j\in \mathcal{J}_{B_i}}\frac{|C_j|}{|\tilde{B_i}|}\delta_{\nu_j}$.
Observe, that the sum of the coefficients $\tfrac{|C_j|}{|\tilde{B_i}|}$ over the indices $j$ from the set
	$$\{j\in \mathcal{J}_{B_i}:\  \nu_j\notin \mathsf{Ball}(\mu_i,\tfrac{\varepsilon}2)\}$$
	is equal to $\xi_i\bigl(\mathcal{M}_{\sigma}(X)\setminus \mathsf{Ball}(\mu_i,\tfrac{\varepsilon}2)\bigr)$. Therefore, by lemma~\ref{barycenter_compact}, it is smaller than $\tfrac{\varepsilon}2$.
	
	 We define the sets $$\mathcal{J}^{\mathsf{NE}}_{B_i}=\{j\in \mathcal{J}_{B_i}:  C_j\notin \U\}, \hspace{20pt} \mathcal{J}^{\mathsf{NE}}_{\mathsf{B}}=\bigcup_{i}\mathcal{J}^{\mathsf{NE}}_{B_i}.$$ In words, $\mathcal{J}_{B_i}^{\mathsf{NE}}$ is the set of indices $j$ such that $C_j$ is a nonergodic block and $I_{C_j}\subset I_{B_i}$. Fix~$j\in\mathcal{J}_{B_i}^{\mathsf{NE}}$. Then $C_j\notin \U$, hence also  $C_j\notin \mathsf{Ball}(\mu_i,\varepsilon)$. Since $d^*(C_j,\nu_j)<\tfrac{\gamma}4<\tfrac{\varepsilon}4$, we~have, $$d^*(\nu_j,\mu_{i})\ge d^*(C_j,\mu_{i})-d^*(C_j,\nu_j)>\varepsilon-\tfrac{\varepsilon}4=\tfrac{3\varepsilon}4>\tfrac{\varepsilon}2.$$ 
	Henceforth, for every $i$, the following inclusion holds
	\begin{equation*}
	\{j\in \mathcal{J}_{B_i}:\  \nu_j\notin \mathsf{Ball}(\mu_i,\tfrac{\varepsilon}2)\} \supset \mathcal{J}^{\mathsf{NE}}_{B_i}.
	\end{equation*} 
	Thus, for every $i$, it is true that
	\begin{equation}\label{oszacowanko}
	\sum_{j\in\mathcal{J}^{\mathsf{NE}}_{B_i}}\frac{|C_j|}{|\tilde{B_i}|}<\frac{\varepsilon}2.
	\end{equation} 
	In words, the fraction of nonergodic blocks $C_j$ in each block $\tilde{B}_i$, is smaller than~$\tfrac{\varepsilon}2$. 
	
	 Let $r$ belong to $\Z$ (resp.\ $\N_0$, for a one-sided subshift) and $N$ belong to $\N$. We denote $\mathcal{I}_r^N=\{i: I_{\tilde{B_i}}\cap [r,r+N)\neq\varnothing\}$. Note that there are at most two blocks $\tilde{B_i}$, for which $I_{\tilde{B_i}}\not\subset [r,r+N)$ and $I_{\tilde{B_i}}\cap[r,r+N)\neq\varnothing$. Thus, $$\sum_{i\in\mathcal{I}_r^N}|\tilde{B_i}|\le N+2n_0.$$ On account of that and by inequality~\eqref{oszacowanko}, we get
	\begin{multline*}
	\frac{1}{N}\Bigl|\bigcup_{j\in \mathcal{J}^{\mathsf{NE}}_{\mathsf{B}}}I_{C_j}\cap [r,r+N)\Bigr|=\frac{1}{N}\sum_{i\in \mathcal{I}_r^N}\sum_{j\in \mathcal{J}^{\mathsf{NE}}_{B_i}}|C_j|= \frac{1}{N}\sum_{i\in \mathcal{I}_r^N}|\tilde{B_i}|\sum_{j\in \mathcal{J}^{\mathsf{NE}}_{B_i}}\frac{|C_j|}{|\tilde{B_i}|}\\< \frac{1}N(N+2n_0)\frac{\varepsilon}{2}=\frac{\varepsilon}{2}+\frac{n_0}{N}\varepsilon.
	\end{multline*}
	Thus, the following inequality holds
	\begin{equation*}
	\sup_{r}\frac{1}{N}\Bigl|\bigcup_{j\in \mathcal{J}^{\mathsf{NE}}_{\mathsf{B}}}I_{C_j}\cap [r,r+N)\Bigr|<\frac{\varepsilon}{2}+\frac{n_0}{N}\varepsilon,
	\end{equation*}
	which implies that \begin{equation}\label{dbC}
	\overline{d}_{\mathsf{Ban}}\Bigl(\bigcup_{j\in \mathcal{J}^{\mathsf{NE}}_{\mathsf{B}}}I_{C_j} \Bigr)\le\frac{\varepsilon}2.
	\end{equation}
	Let $\mathcal{J}^{\mathsf{NE}}=\{j: C_j\notin \U\}$. Our goal is to show that $\overline{d}_{\mathsf{Ban}}\Bigl(\bigcup_{j\in\mathcal{J}^{\mathsf{NE}}}I_{C_j}\Bigr)<\varepsilon$. Obviously, $$\bigcup_{j\in\mathcal{J}^{\mathsf{NE}}}I_{C_j}\subset\bigcup_{j\in\mathcal{J}_{\mathsf{A}}}I_{C_j}\cup \bigcup_{j\in\mathcal{J}^{\mathsf{NE}}_{\mathsf{B}}}I_{C_j}.$$ Therefore, by the inequalities~\eqref{dbA} and~\eqref{dbC}, and subadditivity of upper Banach density, we obtain
	\begin{equation*}
	\overline{d}_{\mathsf{Ban}}\Bigl(\bigcup_{j\in\mathcal{J}^{\mathsf{NE}}}I_{C_j}\Bigr)\le \overline{d}_{\mathsf{Ban}}\Bigl(\bigcup_{j\in\mathcal{J}_{\mathsf{A}}}I_{C_j}\Bigr)+\overline{d}_{\mathsf{Ban}}\Bigl(\bigcup_{j\in\mathcal{J}^{\mathsf{NE}}_{\mathsf{B}}}I_{C_j}\Bigr)<\frac{\varepsilon}{2}+\frac{\varepsilon}{2}=\varepsilon.
	\end{equation*}
\end{proof}
\section{Examples}\label{sekcja3}

 We begin with two examples showing that sometimes the presence of a small fraction of nonergodic blocks in a decomposition as in theorems~\ref{thm1},~\ref{finitistic} and~\ref{unbounded} is inevitable. The first (simple) example shows that for some $x\in\{0,1\}^{\mathbb N_0}$ infinitely many nonergodic blocks must occur. Still, their upper Banach density in this example can be reduced to zero. The example concerns a one-sided symbolic element. A two-sided example can be easily produced by reflection about the coordinate zero.

\begin{example}\label{nontoepliz}
	Define $B_0=01,\ B_1=0011,\ B_2=000111,\dots, B_k=0^{k+1}1^{k+1}$, \dots, and~put
	$$
	x=B_0\ B_1B_0B_1\ B_2B_1B_0B_1B_2\ B_3B_2B_1B_0B_1B_2B_3\ \dots.
	$$
	We let $X$ be the shift-orbit closure of $x$. It is easy to see that $X$ supports only two ergodic measures, $\delta_{\mathbf 0}$ and $\delta_{\mathbf 1}$, the Dirac measures supported at elements $\mathbf 0 = 000\dots$ and $\mathbf 1=111\dots$. Let $\mathcal U$ (restricted to $\mathcal B^*$) be defined by 
	$$
	B\in\mathcal U\ \ \iff \ \ |B|>1\text{ and }\mathrm{Fr}_B(0)\notin[\tfrac15,\tfrac45].
	$$
	It is not hard to see that any block $B$ occurring in $x$ in such a place that it has at least one coordinate common with an ``explicit'' occurrence of $B_0$ (we disregard here the ``implicit'' occurrences of $B_0$ in the centers of the blocks $B_k$, $k>0$), does not belong to $\mathcal U$ (see figure~2). 
	\begin{figure}[ht]
		$$
		\dots|\underset{B_3}{0\,0\,0\,0\,1\,1\,1\,1}|\underset{B_2}{0\,0\,0\,1\,1\,1}|\underset{B_1}{0\,0\,1\,1}|\overset{\mathrm{Fr}(0)=\frac34}{\overbrace{\underset{B_0}{0\,1}|0\,0}}\underset{\!\!\!\!\!\!B_1}{\,1\,1}|\underset{B_2}{0\,0\,0\,1\,1\,1}|\underset{B_3}{0\,0\,0\,0\,1\,1\,1\,1}|\dots
		$$
		\caption{Among all blocks (of lengths larger than 1) having a common coordinate with an explicit $B_0$, the largest frequency of zeros is $\tfrac34$ and is achieved on the block $B_000$. The smallest frequency of zeros is $\tfrac14$ and is achieved on the block $11B_0$.
			Extending these blocks further to the right or left will bring the frequency of zeros closer to $\tfrac12$.}
	\end{figure}
	
	Thus, in any decomposition of $x$ as an infinite concatenation of finite blocks (of lengths bounded or not), there will be infinitely nonergodic components. However, since these components are adjacent to the ``explicit'' occurrences of $B_0$, their upper Banach density can be reduced to zero.
\end{example}

In the next (much more complicated) example, the upper Banach density of nonergodic blocks in theorem~\ref{thm1} cannot be reduced to zero. Since $\mathcal{M}_{\sigma}(X)$ in this example is a Bauer simplex, it also shows that assuming (as in theorem~\ref{thm2}) that $\Me$ is closed does not help reducing to 0 the upper Banach density of nonergodic blocks. The example concerns a two-sided symbolic element. A one-sided example can be easily produced by restriction to the nonnegative coordinates.

\begin{example}\label{toeplitz}
	In this example, $x$ is a binary $\{0,1\}$-valued Toeplitz \sq. The standard construction of such a \sq\ consists in successively filling in periodic patterns of increasing periods until the entire \sq\ is filled. In this particular example, in step~1 we fill periodically two in every six places, as follows:
	$$
	x=\dots\begin{matrix}*&\!\!\!\!*&\!\!\!\!*&\!\!\!\!*&\!\!\!\!0&\!\!\!\!1&\!\!\!\!*&\!\!\!\!*&\!\!\!\!*&\!\!\!\!*&\!\!\!\!0&\!\!\!\!1&\!\!\!\!*&\!\!\!\!*&\!\!\!\!*&\!\!\!\!*&\!\!\!\!\underline0&\!\!\!\!1&\!\!\!\!*&\!\!\!\!*&\!\!\!\!*&\!\!\!\!*&\!\!\!\!0&\!\!\!\!1&\!\!\!\!*&\!\!\!\!*&\!\!\!\!*&\!\!\!\!*&\!\!\!\!0&\!\!\!\!1&\!\!\!\!*&\!\!\!\!*&\!\!\!\!*&\!\!\!\!*&\!\!\!\!\dots\end{matrix}
	$$
	(the stars signify the places left to be filled in the following steps, the zero coordinate is marked by the underlined symbol).
	Abbreviating $0\,1=B_0$ and $*\,*\,*\ *=\bar*$, the above structure of $x$ becomes
	$$
	x=\dots\begin{matrix}\bar*&\!\!\!\!B_0&\!\!\!\!\bar*&\!\!\!\!B_0&\!\!\!\!\bar*&\!\!\!\!B_0&\!\!\!\!\bar*&\!\!\!\!B_0&\!\!\!\!\bar*&\!\!\!\!B_0&\!\!\!\!\bar*&\!\!\!\!B_0&\!\!\!\!\bar*&\!\!\!\!\underline {B_0}&\!\!\!\!\bar*&\!\!\!\!B_0&\!\!\!\!\bar*&\!\!\!\!B_0&\!\!\!\!\bar*&\!\!\!\!B_0&\!\!\!\!\bar*&\!\!\!\!B_0&\!\!\!\!\bar*&\!\!\!\!B_0&\!\!\!\!\bar*&\!\!\!\!B_0&\!\!\!\!\bar*\end{matrix}\dots.
	$$
	In step 2, we fill all four empty spaces in each of the blocks $\bar*$ on either side of every tenth block $B_0$ by placing there the block $B_1=0\,0\,1\,1$, as follows:
	$$
	x=\dots\begin{matrix}\bar*&\!\!\!\!B_0&\!\!\!\!\!\!B_1&\!\!\!\!\!\!\underline {B_0}&\!\!\!\!\!\!B_1&\!\!\!\!\!\!B_0&\!\!\!\!\bar*&\!\!\!\!B_0&\!\!\!\!\bar*&\!\!\!\!B_0&\!\!\!\!\bar*&\!\!\!\!B_0&\!\!\!\!\bar*&\!\!\!\!B_0&\!\!\!\!\bar*&\!\!\!\!B_0&\!\!\!\!\bar*&\!\!\!\!B_0&\!\!\!\!\bar*&\!\!\!\!B_0&\!\!\!\!\bar*&\!\!\!\!B_0&\!\!\!\!\!\!B_1&\!\!\!\!\!\!B_0&\!\!\!\!\!\!B_1&\!\!\!\!\!\!B_0&\!\!\!\!\bar*\end{matrix}\dots.
	$$
	Abbreviating $B_0B_1\!B_0B_1\!B_0=\bar B_1$ and $\bar*\,B_0\,\bar*\,B_0\,\bar*\,B_0\,\bar*\,B_0\,\bar*\,B_0\,\bar*\,B_0\,\bar*\,B_0\,\bar*=\bar{\bar*}$,
	the above structure of $x$ becomes
	$$
	x=\dots\begin{matrix}
	\bar{\bar*}&\!\!\!\!\bar B_1&\!\!\!\!\bar{\bar*}&\!\!\!\!\bar B_1&\!\!\!\!\bar{\bar*}&\!\!\!\!\bar B_1&\!\!\!\!\bar{\bar*}&\!\!\!\!\bar B_1&\!\!\!\!\bar{\bar*}&\!\!\!\!\bar B_1&\!\!\!\!\bar{\bar*}&\!\!\!\!\bar B_1&\!\!\!\!\bar{\bar*}&\!\!\!\!\underline{\bar B_1}&\!\!\!\!\bar{\bar*}&\!\!\!\!\bar  B_1&\!\!\!\!\bar{\bar*}&\!\!\!\!\bar B_1&\!\!\!\!\bar{\bar*}&\!\!\!\!\bar B_1&\!\!\!\!\bar{\bar*}&\!\!\!\!\bar B_1&\!\!\!\!\bar{\bar*}&\!\!\!\!\bar B_1&\!\!\!\!\bar{\bar*}&\!\!\!\!\bar B_1&\!\!\!\!\bar{\bar*}
	\end{matrix}\dots.
	$$
	The block $\bar{\bar*}$ contains $8$ blocks $\bar*$ each having $4$ unfilled places. In step 3, on either side of every $18$th block $\bar B_1$ we fill all the unfilled places in $\bar{\bar*}$ consecutively by the symbols $0\,0\,0\,0\,1\,1\,1\,1$ (repeated 4 times).
In this manner, two out of 18 blocks $\bar{\bar*}$ are replaced by
	$$
	B_2= 0\,0\,0\,0\,B_0\,1\,1\,1\,1\,B_0\,0\,0\,0\,0\,B_0\,1\,1\,1\,1\,B_0\,0\,0\,0\,0\,B_0\,1\,1\,1\,1\,B_0\,0\,0\,0\,0\,B_0\,1\,1\,1\,1.
	$$
	After step 3, $x$ has the following form:
	$$
	x=\dots\begin{matrix}
	\bar{\bar*}&\!\!\!\!\bar B_1&\!\!\!\!\!\!B_2&\!\!\!\!\!\!\underline{\bar B_1}&\!\!\!\!\!\!B_2&\!\!\!\!\!\!\bar B_1&\!\!\!\!\bar{\bar*}&\!\!\!\!\bar B_1&\!\!\!\!\bar{\bar*}&\!\!\!\!\bar B_1&\!\!\!\!\bar{\bar*}&\!\!\!\!\dots&\!\!\!\!\bar{\bar*}&\!\!\!\!\bar B_1&\!\!\!\!\bar{\bar*}&\!\!\!\!\bar B_1&\!\!\!\!\bar{\bar*}&\!\!\!\!\bar B_1&\!\!\!\!\!\!B_2&\!\!\!\!\!\!\bar B_1&\!\!\!\!\!\!B_2&\!\!\!\!\!\!\bar B_1&\!\!\!\!\bar{\bar*}
	\end{matrix}\dots,
	$$
	where, in the central section, $\bar{\bar*}$ occurs 16 times and $\bar B_1$ occurs 15 times.
	Abbreviating $\bar B_1B_2\bar B_1B_2\bar B_1=\bar B_2$ and 
	$$
	\begin{matrix}\begin{matrix}
	\bar{\bar*}&\!\!\!\!\!\bar B_1&\!\!\!\!\!\bar{\bar*}&\!\!\!\!\!
	\bar B_1&\!\!\!\!\!\bar{\bar*}&\!\!\!\!\!\bar B_1&\!\!\!\!\!\bar{\bar*}&\!\!\!\!\!
	\bar B_1&\!\!\!\!\!\bar{\bar*}&\!\!\!\!\!\bar B_1&\!\!\!\!\!\bar{\bar*}&\!\!\!\!\!
	\bar B_1&\!\!\!\!\!\bar{\bar*}&\!\!\!\!\!\bar B_1&\!\!\!\!\!\bar{\bar*}&\!\!\!\!\!
	\bar B_1&\!\!\!\!\!\bar{\bar*}&\!\!\!\!\!\bar B_1&\!\!\!\!\!\bar{\bar*}&\!\!\!\!\!
	\bar B_1&\!\!\!\!\!\bar{\bar*}&\!\!\!\!\!\bar B_1&\!\!\!\!\!\bar{\bar*}&\!\!\!\!\!
	\bar B_1&\!\!\!\!\!\bar{\bar*}&\!\!\!\!\!\bar B_1&\!\!\!\!\!\bar{\bar*}&\!\!\!\!\!
	\bar B_1&\!\!\!\!\!\bar{\bar*}&\!\!\!\!\!\bar B_1&\!\!\!\!\!\bar{\bar*}
	\end{matrix}\end{matrix}=\bar{\bar{\bar*}},
	$$
	the above structure of $x$ becomes
	$$
	x=\dots\begin{matrix}
	\bar{\bar{\bar*}}&\!\!\!\!\bar B_2&\!\!\!\!\bar{\bar{\bar*}}&\!\!\!\!\bar B_2&\!\!\!\!\bar{\bar{\bar*}}&\!\!\!\!\bar B_2&\!\!\!\!\bar{\bar{\bar*}}&\!\!\!\!\bar B_2&\!\!\!\!\bar{\bar{\bar*}}&\!\!\!\!\bar B_2&\!\!\!\!\bar{\bar{\bar*}}&\!\!\!\!\bar B_2&\!\!\!\!\bar{\bar{\bar*}}&\!\!\!\!\underline{\bar B_2}&\!\!\!\!\bar{\bar{\bar*}}&\!\!\!\!\bar B_2&\!\!\!\!\bar{\bar{\bar*}}&\!\!\!\!\bar B_2&\!\!\!\!\bar{\bar{\bar*}}&\!\!\!\!\bar B_2&\!\!\!\!\bar{\bar{\bar*}}&\!\!\!\!\bar B_2&\!\!\!\!\bar{\bar{\bar*}}&\!\!\!\!\bar B_2&\!\!\!\!\bar{\bar{\bar*}}&\!\!\!\!\bar B_2&\!\!\!\!\bar{\bar{\bar*}}
	\end{matrix}\dots.
	$$
	In step number $k$ we will fill two out of $2+2^{k+1}$ blocks $\hat*$ (where \,$\hat{}$\,\  stands for the stack of $k-1$ bars) putting alternatively $2^{k-1}$ zeros and $2^{k-1}$ ones 
	in the consecutive free slots (with this pattern repeated $2^{k-1}2^{k-2}\cdots2^2$ times). We let $\bar B_k$ be the maximal entirely filled continuous block and we let $\bar{\hat*}$ be the (partly unfilled) block between the occurrences of $\bar B_k$. The density of unfilled positions after step $k$ equals $\prod_{i=2}^k\frac{2^{i+1}}{2+2^{i+1}}$ which tends decreasingly to a positive number $d\approx0.63$.
	
	In each step $x$ is positioned so that the zero coordinate falls near the center of an occurrence of $\bar B_k$. Eventually the entire \sq\ $x$ is filled out. Then $x$ is a  bi-infinite Toeplitz \sq\ whose orbit-closure $X$ has the following properties (we skip the standard proofs, see \cite{survey} for an exposition on Toeplitz subshifts):
	\begin{enumerate} 
		\item In every $y\in X$ one can distinguish a periodic part $\mathsf{Per}(y)$ (the 
		positions filled in the construction steps) and the complementary aperiodic part 
		$\mathsf{Aper}(y)$ (the positions filled as a result of closing the orbit of $x$;
		the aperiodic part may be empty). 
		\item For almost every (with respect to any invariant measure on $X$) element $y\in X$, we 
		have $\mathsf{dens}(\mathsf{Aper}(y))=d$ (here $\mathsf{dens}$ denotes the two-sided 
		density of a subset of $\mathbb Z$).
		\item For almost every $y\in X$, $\mathsf{Aper}(y)$ is either entirely filled with zeros 
		or entirely filled with ones.
		\item $X$ carries exactly two ergodic measures: $\mu_0$ and $\mu_1$; $\mu_0$ is 
		supported by such $y\in X$ that $\mathsf{Aper}(y)$ is entirely filled with zeros, $\mu_1$ 
		is supported by such $y\in X$ that $\mathsf{Aper}(y)$ is entirely filled with ones. 
		\item $\mu_0([0])=\frac{1+d}2,\ \mu_0([1])=\frac{1-d}2$ and $\mu_1([0])=\frac{1-d}2,\ 
		\mu_1([1])=\frac{1+d}2$. 
	\end{enumerate}
	
	The last technical thing to observe is that for any $k\ge 1$, any subblock of $\bar B_k$, of length larger than 1, which covers at least one of the two central positions $0\,1$ in $\bar B_k$ has the frequency of zeros ranging between $\tfrac14$ and $\tfrac34$ (see the figure below). 
	\begin{figure}[H]
		$$
		\dots|0\,0\,0\,0|0\,1|1\,1\,1\,1|\underset{B_0}{0\,1}|\underset{B_1}{0\,0\,1\,1}|\overset{\mathrm{Fr}(0)=\tfrac34}{\overbrace{\underset{B_0}{0\,1}|0\,0}}\underset{\!\!\!\!\!\!\!\!B_1}{1\,1}|\underset{B_0}{0\,1}|0\,0\,0\,0|0\,1|1\,1\,1\,1|\dots
		$$
		\caption{The figure shows the central part of $\bar B_k,\ k\ge 2$. Among all subblocks (of lengths larger than 1) of $\bar B_k$, having a common coordinate with the central $B_0$, the largest frequency of zeros is $\tfrac34$ and is achieved on the block $B_000$. The smallest frequency of zeros is $\tfrac14$ and is achieved on the block $11B_0$. Extending these blocks further to the right or left will only bring the frequency of zeros closer to $\frac12$.}
	\end{figure}
	
	Let us define $\mathcal U$ (restricted to $\mathcal B^*$) by the properties $|B|>1$ and $\mathrm{Fr}_B(0)\notin[\tfrac15,\tfrac45]$.
	Because $\frac{1-d}2<\tfrac15$ and $\frac{1+d}2>\tfrac45$, $\mathcal U$ is an open neighborhood of $\M^{\mathsf{erg}}_\sigma(X)$. Suppose $x$ is represented as an infinite concatenation of some blocks $C_j$ ($j\in\mathbb Z$) of lengths bounded by some $n$. Let $k$ be such that $n<\frac12|\bar B_k|$. Then, in every occurrence of $\bar B_k$ there is a block $C_j$ not disjoint with the central $B_0$, and this block is entirely covered by the $\bar B_k$. As we have noted above, either $|C_j|=1$ or $\mathrm{Fr}_{C_j}(0)\in[\frac14,\frac34]$. In either case $C_j\notin\mathcal U$, i.e.\ $C_j$ is a nonergodic block. We have shown that each occurrence of $\bar B_k$ in $x$ contains a nonergodic block $C_j$. Since the explicit occurrences of $\bar B_k$ are periodic, all occurrences of $\bar B_k$ have positive lower Banach density. This implies that nonergodic blocks $C_j$ have positive lower Banach density as well.
\end{example}
\begin{remark}
The block $B_k$ in the above example shows also that in theorem~\ref{finitistic} the presence of nonergodic subblocks is inevitable in any decomposition of a long block.
\end{remark}

 We end this section with an example showing that the assumption of compactness of the set $\Me[X]$ in theorem~\ref{thm2} is essential: there exist a subshift $X$ with $\mathcal{M}_{\sigma}(X)$ not being a Bauer simplex and an element $x\in X$ such that, for any $n\ge 1$, there exists a representation of $x$ as a concatenation of blocks with lengths bounded by $n$, such that the upper Banach density of nonergodic blocks equals 1. The example concerns a one-sided subshift. A two-sided example can be obtained by reflection about the coordinate zero.
\begin{example}
		Let $(B_k)_{k\in\N}$ be the sequence of blocks defined as follows: $B_1=111000$, $B_2=111111000000,\dots$, $B_k=1^{3k}0^{3k},\ldots$. Let $x\in\{0,1\}^{\N_0}$ be the following concatenation:
	\begin{equation*}
	x=B_1\ B_1B_1B_2B_2\ B_1B_1B_1B_2B_2B_2B_3B_3B_3\cdots\underbrace{B_{1}...B_{1}}_{k\text{ times}}\underbrace{B_{2}...B_{2}}_{k\text{ times}}\dots\underbrace{B_{k}...B_{k}}_{k\text{ times}}\cdots
	\end{equation*}
	
	 Let $X$ be the orbit closure of $x$. It is easy to check that the only ergodic measures on $X$ are $\delta_{\bm{0}}$ and $\delta_{\bm{1}}$ and the measures $\mu_k$, $k\in\N$ supported on the periodic orbits of the points $x_k=B_kB_k\dots$. Observe that the sequence $(\mu_k)_{k\in\N}$ converges in the weak-$\star$ topology to the measure $\tfrac{1}{2}(\delta_{\bm{0}}+\delta_{\bm{1}})\notin \Me[X]$. Consequently, the set $\Me[X]$ is not closed in $\mathcal{M}(X)$. Moreover, every measure $\mu\in\Me[X]$ satisfies $\mu([0])\in\{0,\tfrac{1}2,1\}$. Hence, the condition $\mathrm{Fr}_{B}(0)\in[0,\tfrac{1}5)\cup (\tfrac{2}5, \tfrac{3}5)\cup (\tfrac{4}5,1]$ defines an open neighbourhood $\mathcal{U}$ of $\Me[X]$ (restricted to the set $\mathcal{B}^*(X)$).
	
	 For every $m\in\N$, each number $n\ge (m-1)m$ can be written as a combination $n=am+b(m+1)$, where $a,b\in\N$. For a fixed $k\in\N$, let $n_1\ge(3k-1)3k$ be an initial coordinate of a series of repetitions of the blocks $B_k$ in $x$, and denote by $m_1$ the terminal coordinate of that series. Because $n_1+k-1\ge(3k-1)3k$, we can decompose $x|_{[0,n_1+k)}$ into blocks $C_j$ of lengths $3k$ or $3k+1$. Then, we divide $x|_{[n_1+k,m_1-2k]}$ into blocks~$C_j$ of lengths equal to $3k$. Observe that these blocks have the form \begin{equation}\label{3}
	C_j=\underbrace{11...1}_{k}\underbrace{000...000}_{2k}\ \  \text{or} \ \ C_j=\underbrace{00...0}_{k}\underbrace{111...111}_{2k},
	\end{equation}
	hence $\mathrm{Fr}_{C_j}(0)\in\{\tfrac13,\tfrac23\}$. Therefore, these blocks $C_j$ are nonergodic.
	Let $n_2\ge m_1+(3k-1)3k$ and $m_2$ be the initial and terminal coordinates of another series of repetitions of the block $B_k$ (note that this series is longer than the preceding one). It is possible to divide $x|_{[m_1-2k+1,n_2+k)}$ into blocks $C_j$ of lengths $3k$ or $3k+1$. Then we decompose $x|_{[n_2+k,m_2-2k]}$ into blocks $C_j$ of lengths $3k$. These blocks also have the form~\eqref{3}, hence are nonergodic. We continue the construction similarly for $s\ge 3$. Since ${(m_l-2k)-(n_s+k)\to +\infty}$ as $s\to+\infty$, the upper Banach density of the nonergodic blocks $C_j$ is equal to 1.
\end{example}
\section{Symbolic systems with an action of a countable amenable group}\label{sekcja4}
\subsection{Amenable groups.}\label{podsekcja41}

 In what follows, $G$ denotes a countable (infinite), discrete group. All theorems provided in this subsection are standard and their proofs will be omitted.
\begin{definition}
	Fix an $\varepsilon>0$ and a finite subset $K\subset G$. A finite subset $F\subset G$ is called $(K,\varepsilon)$-\emph{invariant} if it satisfies
	\begin{equation}
	\frac{|F\triangle KF|}{|F|}<\varepsilon.
	\end{equation}
If $K=\{g\}$ for some $g\in G$ then we say that $F$ is $(g,\varepsilon)$-invariant.
\end{definition}
\begin{fact}\label{gandK} Let $\varepsilon$ be a positive number and let $K,F$ be finite subsets of $G$.
	\begin{enumerate}
		\item [a)] If, for every $g\in K$, the set $F$ is $(g,\tfrac{\varepsilon}{|K|})$-invariant then $F$ is $(K,\varepsilon)$-invariant.
		\item [b)] If $F$ is $(K,\varepsilon)$-invariant then, for every $g\in K$, $F$ is $(g,2\varepsilon)$-invariant. 
	\end{enumerate}
\end{fact}
\begin{definition}
	A set $A\subset G$ is called an $\varepsilon$-modification of a finite set $B\subset G$ if
	\begin{equation}
	\frac{|A\triangle B|}{|B|}<\varepsilon.
	\end{equation}
	If $A$ is an $\varepsilon$-modification of $B$ then $B$ is an $(\tfrac{\varepsilon}{1-\varepsilon})$-modification of $A$.
\end{definition}
\begin{definition}\label{folner}
	By a \emph{F{\o}lner sequence} we mean a sequence $(F_n)_{n\in\N}$ of finite subsets of $G$, such that, for every $\varepsilon>0$ and every finite set $K\subset G$, the sets $F_n$ are eventually (i.e.\ for $n$ large enough) $(K,\varepsilon)$-invariant.
\end{definition}
\begin{remark}
	Since $G$ is infinite, any F{\o}lner sequence satisfies $\lim\limits_{n\to+\infty}|F_n|=+\infty$.
\end{remark}
\begin{definition}
	A countable discrete group $G$ is called \emph{amenable} if it~has~a~F{\o}lner~sequence.
\end{definition}
\noindent For other definitions of amenability and the proofs of their equivalence, see e.g.~\cite{Namioka}.
\begin{definition}\label{densities_G}
	Let $\mathcal{F}(G)$ denote the collection of all finite subsets of a countable group~$G$ and let $A$ be any subset of $G$. \emph{Upper Banach density} of $A$ is defined as
	\begin{equation}\label{upper_banach_densityG}
	\overline{d}_{\mathsf{Ban}}(A)=\inf_{F\in\mathcal{F}(G)}\sup_{g\in G}\frac{|A\cap Fg|}{|F|}.
	\end{equation}
	Similarly, we define \emph{lower Banach density} as
	\begin{equation}\label{lower_banach_densityG}
	\underline{d}_{\mathsf{Ban}}(A)=\sup_{F\in\mathcal{F}(G)}\inf_{g\in G}\frac{|A\cap Fg|}{|F|}.
	\end{equation}
	If $\overline{d}_{\mathsf{Ban}}(A)=\underline{d}_{\mathsf{Ban}}(A)$ then we denote the common value by $d_{\mathsf{Ban}}(A)$ and call it \emph{Banach density} of $A$.
\end{definition}
\begin{remark}\label{densities_comp}
	Similarly as for subsets of $\Z$, we have $\overline{d}_{\mathsf{Ban}}(A)=1-\underline{d}_{\mathsf{Ban}}(A^{\mathsf{c}})$ for any $A\subset G$. Upper Banach density is subadditive, i.e.\ for every $A,B\subset G$ we have
	\begin{equation*}
	\overline{d}_{\mathsf{Ban}}(A\cup B)\le \overline{d}_{\mathsf{Ban}}(A)+\overline{d}_{\mathsf{Ban}}(B).
	\end{equation*}
\end{remark}

 Using the notion of a F{\o}lner sequence we can provide equivalent formulas for upper and lower Banach densities in countable amenable groups. In the case of $G=\Z$ and $F_n=[0,n)$, $n\in\N$, the formulas below coincide with those in definition~\ref{densities}. The proof of the following theorem is provided e.g.\ in \cite[lemma 2.9]{DHZ}.
\begin{theorem}\label{density_foelner}
	Let $G$ be a countable amenable group and let $(F_n)_{n\in\N}$ be a F{\o}lner sequence in $G$. Then
	\begin{eqnarray}
	\overline{d}_{\mathsf{Ban}}(A)=\lim\limits_{n\to+\infty}\sup_{g\in G}\frac{|A\cap F_ng|}{|F_n|},\\
	\underline{d}_{\mathsf{Ban}}(A)=\lim\limits_{n\to+\infty}\inf_{g\in G}\frac{|A\cap F_ng|}{|F_n|}.	
	\end{eqnarray}
\end{theorem}

  By a topological dynamical system with an action of a group $G$ we mean a pair $(X,\tau)$, where $X$ is a compact metric space and $\tau$ is a homomorphism from $G$ to the group of all self-homeomorphisms of $X$ with the operation of composition. For brevity, we will write $g(x)$ instead of $(\tau(g))(x)$. As before, we denote by $\M(X)$ the set of all Borel probability measures on $X$, by $\mathcal{M}_{\tau}(X)\subset \M(X)$ we denote the set of all $\tau$-invariant measures on $X$ (i.e. measures which are $g$-invariant for every $g\in G$) and by $\M_{\tau}^{\mathsf{erg}}(X)\subset \M_{\tau}(X)$ we denote the set of all ergodic measures (i.e. measures such that $\mu(A)\in\{0,1\}$ for all $\tau$-invariant subsets $A\subset X$).
\begin{theorem}\label{inv_measureG}
	Let $G$ be a countable amenable group and let $(F_n)_{n\in\N}$ be a F{\o}lner sequence in $G$. Fix a topological dynamical system $(X,\tau)$ with an action $\tau$ of $G$. Let $(\nu_n)_{n\in\N}$ be a sequence of Borel probability measures on $X$. We define the sequence of measures $\mu_n$, by 
	\begin{equation*}
	\mu_n=\frac{1}{|F_n|}\sum_{g\in F_n}g(\nu_n),
	\end{equation*}
	where $\bigl(g(\nu_n)\bigr)(A)=\nu_n(g^{-1}(A))$ for every Borel set $A\subset X$. Then $(\mu_n)_{n\in\N}$ has a subsequence converging, in the weak-$\star$ topology, to a $\tau$-invariant measure $\mu$.
\end{theorem}
\begin{corollary}
	If $(X,\tau)$ is a topological dynamical system with an action of a countable amenable group $G$, then the set $\M_{\tau}(X)$ is nonempty.
\end{corollary}

 Now we formulate a generalization of lemma~\ref{upper_banach_density_invariant}.
\begin{lemma}\label{upper_banach_density_invariantG}
	Let $(X,\tau)$ be a topological dynamical system with an action of a countable amenable group $G$ and let $D\subset X$ be a closed set. The following inequality holds
	\begin{equation*}
	\sup_{\mu\in\mathcal{M}_{\tau}(X)}\mu(D)\ge \sup_{x\in X}\overline{d}_{\mathsf{Ban}}(\{g\in G: g(x)\in D\}).
	\end{equation*}
\end{lemma}
\subsection{G-subshifts.}\label{podsekcja42}

 Let $G$ be a countable group and let $\Lambda$ be a finite, discrete space (an alphabet). Let us consider the space $\Lambda^G$. For every $g\in G$ we define the transformation $\sigma(g): \Lambda^G\to \Lambda^G$ given by
\begin{equation}\label{shiftG}
\bigl((\sigma(g))(x)\bigr)(h)=x(hg), \ h\in G.
\end{equation}
Then $\sigma$ is an action of $G$ on $\Lambda^G$. As in the previous subsection we will write $g(x)$ instead of $(\sigma(g))(x)$. By a symbolic system with the action of $G$ (a $G$-subshift) we mean any set $X\subset \Lambda^G$, which is closed and $\sigma$-invariant.

 Let $K\subset G$ be a finite set. By a \emph{block with the domain $K$} we mean an element $C\in\Lambda^K$. For two blocks $C\in\Lambda^K$ and $C'\in\Lambda^{Kg}$ for some $g\in G$ we write $C\approx C'$ if for every $k\in K$ we have $C(k)=C'(kg)$. If for some $x\in X$ and $g\in G$ we have $x|_{Kg}\approx C$, then we say that the block $C$ \emph{occurs in $x$}. By $\mathcal{B}^*(X)$ we denote the set of all finite blocks occurring in points $x\in X$. Similarly, if for a finite set $F\subset G$ and an element $B\in\Lambda^F$, there exists $g\in G$ such that $Kg\subset F$ and $B|_{Kg}\approx C$, then we say that the block $C$ occurs in $B$.
\begin{definition}
	Let $K,F\subset G$ be finite sets. The $K$-\emph{core} of $F$ is the set
	\begin{equation}
	F_K=\{g\in F: Kg\subset F \}.
	\end{equation}
\end{definition}

 The following property of a $K$-core will be useful later in section~\ref{sekcja5} (for the proof see e.g. \cite[lemma 2.6]{DHZ}).
\begin{lemma}\label{lemma-core}
	Let $K,F\subset G$ be finite sets. If $F$ is $(K,\varepsilon)$-invariant, then $\frac{|F_K|}{|F|}>1-\varepsilon|K|$.
\end{lemma}

 Using the notion of a core we can define the frequency of occurrences of one block in another.
\begin{definition}
	Let $K,F\subset G$ be finite sets and let $C\in\Lambda^K$, and $B\in\Lambda^F$. The \emph{frequency of occurrences of the block $C$ in the block $B$} is the number
	\begin{equation}
	\mathrm{Fr}_{B}(C)=\frac{|\{g\in F_K: B|_{Kg}\approx C\}|}{|F|}.
	\end{equation}
\end{definition}
\begin{remark}
	If for finite sets $K,F\subset G$ and every $g\in F$ we have $Kg\not\subset F$, then $F_K=\varnothing$, hence for any $C\in\Lambda^K$ and $B\in\Lambda^F$ we have $\mathrm{Fr}_B(C)=0$.
\end{remark}

 At this point we enumerate a collection $\mathcal{F}(G)$ of all (countably many) finite subsets of~$G$, getting a sequence $(K_l)_{l\in\N}$. With the help of this sequence we can define a pseudometric on the set $\mathcal{B}^*(X)$, as follows
\begin{equation}\label{metric_blocksG}
d^*(B_1,B_2)=\sum_{l=1}^{+\infty}\frac{2^{-l}}{(|K_l|+1)}\sum_{C\in\Lambda^{K_l}}|\mathrm{Fr}_{B_1}(C)-\mathrm{Fr}_{B_2}(C)|, \ \ B_1,B_2\in\mathcal{B}^*(X).
\end{equation}
Observe that $d^*(B,B')=0$ if and only if $B\approx B'$. For every finite set $K\subset G$, with each block $C\in\Lambda^K$ we associate the cylinder $[C]=\{x\in X: x|_K=C\}$. Since characteristic functions of cylinders associated with blocks are linearly dense in the Banach space of all continuous functions on $X$, the following formula 
\begin{equation}\label{metric_measuresG}
d^*(\mu_1,\mu_2)=\sum_{l=1}^{+\infty}\frac{2^{-l}}{(|K_l|+1)}\sum_{C\in\Lambda^{K_l}}|\mu_1([C])-\mu_2([C])|, \ \ \mu_1,\mu_2\in\M(X)
\end{equation}
defines a metric on $\mathcal{M}(X)$, compatible with the weak-$\star$ topology. We also define a distance between a block and a measure by
\begin{equation}\label{metric_block_measureG}
d^*(B,\mu)=\sum_{l=1}^{+\infty}\frac{2^{-l}}{(|K_l|+1)}\sum_{C\in\Lambda^{K_l}}|\mathrm{Fr}_{B}(C)-\mu([C])|, \ \ B\in\mathcal{B}^*(X),\mu\in\M(X).
\end{equation}
The equations~\eqref{metric_blocksG},~\eqref{metric_measuresG} and~\eqref{metric_block_measureG} define a pseudometric on the set $\mathcal{B}^*(X)\cup\M(X)$, which is a metric on the set $\bigl(\mathcal{B}^*(X)/_{\approx}\bigr)\cup\M(X)$. The following theorem is a straightforward generalization of theorem~\ref{blocks_close_to_measures}:
\begin{theorem}\label{blocks_close_to_measuresG}
	Let $(F_n)_{n\in\N}$ be a F{\o}lner sequence in a countable amenable group $G$. For every $\varepsilon>0$ there exists $n_0\in\N$ such that for every $n\ge n_0$ and every block $B\in \Lambda^{F_n}$ occurring in some $x\in X$, the following inequality holds
	\begin{equation*}
	d^*(B,\M_{\sigma}(X))<\varepsilon.
	\end{equation*}
\end{theorem}

 The next lemma will be used in section~\ref{sekcja5}. Although it is intuitively obvious, its rigorous proof is unexpectedly technical, hence we provide it whole.
\begin{lemma}\label{close_mod}
	If $F,H\subset G$ are finite and $H$ is an $\varepsilon$-modification of $F$, $\varepsilon>0$, then, for every $x\in X$, we have
	\begin{equation*}
	d^*(x|_F,x|_H)<\varepsilon.
	\end{equation*}
\end{lemma}
\begin{proof}
	Let $K\subset G$ be a finite set. First, we estimate the cardinality of the set $F_K\triangle H_K$. Observe that $F_K\triangle H_K\subset K^{-1}(F\triangle H)$. Indeed, if $g\in F_K\setminus H_K$ then $kg\in F$ for all $k\in K$ and there exists $k_0\in K$ satisfying $k_0g\notin H$, which implies $g\in k_0^{-1}(F\setminus H)\subset K^{-1}(F\setminus H)$ (by symmetry, $g\in H_K\setminus F_K$ implies $g\in K^{-1}(H\setminus F)$). Thus, $|F_K\triangle H_K|\le |K|\cdot|F\triangle H|$.
	
	 An occurrence $x|_{Kg}\approx C$ of a block $C\in\Lambda^K$ in $x$ is accounted in the computation of $\mathrm{Fr}_{x|_F}(C)$ and not accounted in the computation of $\mathrm{Fr}_{x|_H}(C)$, or vice versa, if and only if $g\in F_K\triangle H_K$. Henceforth, the following inequality holds
	\begin{equation*}
	\sum_{C\in\Lambda^{K}}\bigl||H|\mathrm{Fr}_{x|_H}(C)-|F|\mathrm{Fr}_{x|_F}(C)\bigr|\le |K||F\triangle H|<|K||F|\varepsilon.
	\end{equation*}
Therefore, we obtain
	\begin{multline*}
	d^*(x|_F,x|_H)=\sum_{l=1}^{+\infty}\frac{2^{-l}}{|K_l|+1}\sum_{C\in\Lambda^{K_l}}\bigl|\mathrm{Fr}_{x|_H}(C)-\mathrm{Fr}_{x|_F}(C)\bigr|\\\le \sum_{l=1}^{+\infty}\frac{2^{-l}}{|K_l|+1}\sum_{C\in\Lambda^{K_l}}\Bigl(\Bigl|\mathrm{Fr}_{x|_H}(C)-\frac{|H|}{|F|}\mathrm{Fr}_{x|_H}(C)\Bigr|+\frac{1}{|F|}\Bigl||H|\mathrm{Fr}_{x|_H}(C)-|F|\mathrm{Fr}_{x|_F}(C)\Bigr|\Bigr)\\\le \sum_{l=1}^{+\infty}\frac{2^{-l}}{|K_l|+1}\Bigl(\Bigl|1-\frac{|H|}{|F|}\Bigr|+|K_l|\varepsilon\Bigr).
	\end{multline*}
	Since $H$ is an $\varepsilon$-modification of $F$, we have $|1-\tfrac{|H|}{|F|}|<\varepsilon$, implying ${d^*(x|_F,x|_H)<\varepsilon}$.
\end{proof}

 We finish this subsection with the definition of an $(F_n)$-generic element for an invariant measure, which, in case $\Z=G$ and $F_n=[0,n)$, reduces to definition~\ref{generic}.
\begin{definition}\label{genericG}
	Let $(F_n)_{n\in\N}$ be a F{\o}lner sequence in a countable amenable group $G$. A symbolic element $x\in X\subset \Lambda^G$ is called $(F_n)$-\emph{generic} ($(F_n)$-\emph{quasigeneric}) for a measure $\mu\in\mathcal{M}_{\sigma}(X)$ if the sequence (some subsequence of the sequence) of the blocks $(x|_{F_n})_{n\in\N}$ converges to the measure $\mu$ in the pseudometric $d^*$ on $\mathcal{B}^*(X)\cup\M(X)$.
\end{definition}
\subsection{Tilings of countable amenable groups.}\label{podsekcja43}

 The aim of this subsection is to provide the necessary background concerning the theory of tilings, playing an instrumental role in generalizations of theorems from section~\ref{sekcja2} to the case of countable amenable groups.
\begin{definition}
Let $G$ be a countable group. A \emph{tiling} is a partition $\T$ of $G$ into finite, pairwise disjoint subsets $T\in\T$ (called the \emph{tiles}), such that there exists a finite collection $\mathcal{S}$ (called the \emph{collection of shapes}) of finite sets $S$ (not necessarily all of them different), each of them containing the unit $e$ of $G$, such that every $T\in \T$ has a form $T=Sc$ for some $S\in\mathcal{S}$ and $c\in T$.
\end{definition}

 Given a tiling $\T$, for every tile $T\in\T$ we choose one pair $(S,c)$, where $S\in\mathcal{S}$ and $c\in T$ such that $T=Sc$. We call $S$ the \emph{shape of the tile $T$} and $c$ the \emph{center of the tile $T$}. Every tiling $\T$ can be represented as a symbolic element (also denoted by the same letter $\T$) over the alphabet $V=\{``S": S\in \mathcal{S}\}\cup\{``0"\}$, as~follows
\begin{equation*}
\T(g)=\begin{cases}
\text{``$S$''},& \text{if g is a center of a tile with the shape $S$}, \\
\text{``$0$''},& \text{otherwise}.
\end{cases}
\end{equation*}
\begin{definition}
	Let $\mathcal{S}$ be a collection of shapes and let $V=\{``S": S\in \mathcal{S}\}\cup\{``0"\}$. A~\emph{dynamical tiling} is a closed and shift-invariant set $\mathsf{T}\subset V^G$, consisting of tilings.
\end{definition}

 Needless to say, the orbit closure of any tiling $\T$ is a dynamical tiling.
\begin{definition}
	Let $(\mathsf{T}_{\!k})_{k\in\N}$ be a sequence of dynamical tilings. A \emph{system of dynamical tilings} is a topological joining $(\bT,\bm{\sigma})=\bigvee_{k\in\N}(\mathsf{T}_{\!k},\sigma)$, i.e.\ $\bT$ is a closed, $\bm{\sigma}$-~invariant subset of the product $\prod_{k\in\N}\mathsf{T}_{\!k}$, where $\bm{\sigma}$ is defined by $\bigl(\bm{\sigma}(g)\bigr)(\T_1,\T_2,\dots)=\bigl((\sigma(g))(\T_1),(\sigma(g))(\T_2),\dots\bigr)$. For brevity, a system of dynamical tilings will be sometimes denoted by $\bT=\bigvee_{k\in\N}\mathsf{T}_{\!k}$ and instead of $\bigl(\bm{\sigma}(g)\bigr)(\bm{\T})$ we will write $g(\bm{\T})$, $g\in G$.
\end{definition}
\begin{definition}
	Let $\bT=\bigvee_{k\in\N}\mathsf{T}_{\!k}$ be a system of dynamical tilings of $G$ and let $\mathcal{S}_k$ denote the collection of shapes of $\mathsf{T}_{\!k}$. We say that the system of tilings $\bT$ is:
	\begin{enumerate}
		\item [1)] \emph{F{\o}lner}, if the collection of shapes $\bigcup_{k\in\N}\mathcal{S}_k$ arranged in a sequence is a F{\o}lner sequence;
		\item [2)] \emph{congruent}, if for every $\bm{\T}=(\T_k)_{k\in\N}\in\bT$ and each $k\in\N$, every tile $T\in\T_{k+1}$ is a union of some tiles of $\T_k$;
		\item [3)] \emph{deterministic}, if it is congruent and for every $k\in\N$ and every shape $S'\in\mathcal{S}_{k+1}$, there exist sets $C_S(S')$ indexed by the shapes $S\in\mathcal{S}_k$, such that
		\begin{equation*}
		S'=\bigcup_{S\in\mathcal{S}_k}\bigcup_{c\in C_S(S')}Sc,
		\end{equation*}
		and for each $\bm{\T}=(\T_l)_{l\in\N}\in\bT$, if $S'c'$ is a tile of $\T_{k+1}$, then for every $S\in\mathcal{S}_k$ and $c\in C_S(S')$, the set $Scc'$ is a tile of $\T_k$.
	\end{enumerate}
\end{definition}
\noindent A deterministic system $\bT$ of dynamical tilings has the property that for every  $\bm{\T}=(\T_k)_{k\in\N}\in\bT$ and $k\in \N$, each tiling $\T_{k}$ uniquely determines the tilings $\T_{k'}$ for $k'\le k$.

 The following useful theorem can be found in \cite[theorem~5.2]{DHZ}.
\begin{theorem}
	For every countable amenable group $G$ there exists a F{\o}lner, deterministic system of dynamical tilings of $G$.
\end{theorem}

 We finish this section by providing a simplified version of \cite[lemma~3.4]{DHZ} and \cite[lemma~4.15]{DZ}, concerning the lower Banach density in the context of a fixed tiling $\T$.
\begin{lemma}\label{density_tiling}
	Let $\T$ be a tiling of a countable group $G$. If a subset $A\subset G$ satisfies $\frac{|T\cap A|}{|T|}\ge1-\varepsilon$, for every tile $T\in\T$ and some $\varepsilon>0$, then $\underline{d}_{\mathsf{Ban}}(A)\ge1-\varepsilon$.
\end{lemma}
\section{Decomposition of a symbolic element over $G$ into ergodic blocks}\label{sekcja5}

 This section contains generalizations of theorems from section~\ref{sekcja2} to the case of symbolic systems with the action of a countable amenable group. In what follows, $G$ denotes a countable amenable group, $(X,\sigma)$ denotes a symbolic system with the shift action $\sigma$ of $G$ and $\bT=\bigvee_{k\in\N}\mathsf{T}_{\!k}$ is a F{\o}lner, deterministic system of dynamical tilings of $G$. We let $\mathcal{S}_k$ denote the collection of shapes of $\mathsf{T}_k$. We define $\bm{X}=X\times \bT$. On the space $\bm{X}$ we will consider actions of two groups, $G\times G$, given by $(g,h)(x,\bm{\T})=(g(x),h(\bm{\T}))$, and of $G$, given by $g(x,\bm{\T})=(g(x),g(\bm{\T}))$. By $\M_{(G\times G)}(\bm{X})$ we will denote the set of $(G\times G)$-invariant measures, i.e.\ measures on $\bm{X}$ which are $(g,h)$-invariant for every $(g,h)\in G\times G$, whereas $\M_{G}(\bm{X})$ will stand for the set of $G$-invariant measures, i.e.\ measures on $\bm{X}$, which are $(g,g)$-invariant for every $g\in G$.

 For a fixed $\bm{\T}=(\T_k)_{k\in\N}\in\bT$ and $g\in G$, by $T_k^g(\bm{\T})$ we will denote the unique tile belonging to $\T_k$, such that $g\in T$. In particular, by $T_k^e(\bm{\T})$ we will denote the \emph{central} tile $T\in\T_k$ containing the unit $e$. The simplified notation $T_k^g$ in place of $T_k^g(\bm{\T})$ always refers to the last sequence of tilings $\bm{\T}$ mentioned in the text prior to the discussed $T_k^g$. 

 We begin with a series of lemmas. The first of them concerns disintegrations of $(G\times G)$-invariant measures. For details of the theory of disintegration of measures, we refer the reader e.g. to \cite{Helson}.
\begin{lemma}\label{disint}
	If $\mu$ is a $(G\times G)$-invariant measure on $\bm{X}$ and $\{\mu_{\bm{\T}}: \bm{\T}\in \bT\}$ is a disintegration of $\mu$ with respect to the marginal measure $\mu_{\bT}$ on $\bT$, then for $\mu_{\bT}$-almost every $\bm{\T}\in\bT$, $\mu_{\bm{\T}}$ is a $\sigma$-invariant measure on $X$.
\end{lemma}
\begin{proof}
	Let $g\in G$ be fixed. By the definition of a disintegration of a measure, for every measurable function $\Phi$ on $\bm{X}$ we have
	\begin{equation*}
	\int_{X\times\bT}\Phi(x,\bm{\T})\mathrm{d}\mu(x,\bm{\T})=\int_{\bT}\int_X\Phi(x,\bm{\T})\mathrm{d}\mu_{\bm{\T}}(x)\mathrm{d}\mu_{\bT}(\bm{\T}).
	\end{equation*}
Using the $(G\times G)$-invariance of $\mu$, we obtain
	\begin{multline*}
	\int_{X\times\bT}\Phi(x,\bm{\T})\mathrm{d}\mu(x,\bm{\T})=\int_{X\times\bT}\Phi(g(x),e(\bm{\T}))\mathrm{d}\mu(x,\bm{\T})\\=\int_{\bT}\int_X\Phi(g(x),\bm{\T})\mathrm{d}\mu_{\bm{\T}}(x)\mathrm{d}\mu_{\bT}(\bm{\T})=\int_{\bT}\int_X\Phi(y,\bm{\T})\mathrm{d}\mu_{\bm{\T}}(g^{-1}(y))\mathrm{d}\mu_{\bT}(\bm{\T})\\=\int_{\bT}\int_X\Phi(y,\bm{\T})\mathrm{d}(g(\mu_{\bm{\T}}))(y)\mathrm{d}\mu_{\bT}(\bm{\T}).
	\end{multline*}
	We have shown that $\bm{\T}\mapsto g(\mu_{\bm{\T}})$ is also a disintegration of $\mu$. By uniqueness of the disintegration, the equality
	$\mu_{\bm{\T}}=g(\mu_{\bm{\T}})$ holds for $\mu_{\bT}$-almost every $\bm{\T}$. Since there are countably many elements $g\in G$, for $\mu_{\bT}$-almost every $\bm{\T}$ the measure $\mu_{\bm{\T}}$ is $\sigma$-invariant.
\end{proof}
\par The next two lemmas are analogs of lemmas~\ref{Dum} and~\ref{Eumn} from section~\ref{sekcja2}.
\begin{lemma}\label{measure_zero}
	Let $\mathcal{U}\supset \M_{\sigma}^{\mathsf{erg}}(X)$ be an open set in $\mathcal{B}^*(X)\cup\mathcal{M}_{\sigma}(X)$ and let $m\in\N$. The set
	\begin{equation*}
	\bm{D}_{\U,m}=\left\{(x,\bm{\T})\in \bm{X}: \forall_{k\ge m}\hspace{3pt} x|_{T_k^e}\notin \mathcal{U}\right\}
	\end{equation*}
	is a null set for every $(G\times G)$-invariant measure on $\bm{X}$.
\end{lemma}
\begin{proof} It is not hard to see that $\bm{D}_{\U,m}$ is a Borel (in fact clopen) subset of $\bm{X}$. Suppose that for some $(G\times G)$-invariant measure $\mu$ on $\bm{X}$, we have $\mu(\bm{D}_{\U,m})>0$. Let $\{\mu_{\bm{\T}}: \bm{\T}\in \bT\}$ be the disintegration of the measure $\mu$ with respect to $\mu_{\bT}$. By lemma~\ref{disint}, $\mu_{\bT}$-almost all measures $\mu_{\bm{\T}}$ are $\sigma$-invariant. Thus, the following holds:
	\begin{equation*}
	0<\mu(\bm{D}_{\U,m})=\int_{\bT}\int_{X}\bm{1}_{\bm{D}_{\U,m}}(x,\bm{\T})\mathrm{d}\mu_{\bm{\T}}(x)\mathrm{d}\mu_{\bT}(\bm{\T}).
	\end{equation*}
	Hence, there exists $\bm{\T}\in\bT$ such that both
	\begin{equation*}
	\int_{X}\bm{1}_{\bm{D}_{\U,m}}(x,\bm{\T})\mathrm{d}\mu_{\bm{\T}}(x)>0
	\end{equation*}
	and the measure $\mu_{\bm{\T}}$ is $\sigma$-invariant. We have shown that the set $$D_{\U,m,\bm{\T}}=\bigl\{x\in X: \forall_{k\ge m}\hspace{3pt} x|_{T_k^e}\notin \mathcal{U}\bigr\}$$ has positive measure for a $\sigma$-invariant measure. Thus, there also exists an ergodic measure $\mu_0$ on $X$ such that $\mu_0(D_{\U,m,\bm{\T}})>0$. Hence, by the ergodic theorem, there exists an element $x\in D_{\U,m,\bm{\T}}$, which is quasigeneric\footnote{The existence of an $(F_n)$-quasigeneric element for an ergodic measure can be deduced from Lindenstrauss' ergodic theorem (see \cite{Lindenstrauss}). It also follows from the much more elementary mean ergodic theorem combined with the fact that any sequence of functions convergent in measure has an almost-everywhere convergent subsequence.} for $\mu_0$, along the F{\o}lner sequence $(T^e_k)_{k\in\N}$ consisting of the central tiles of $\bm{\T}$. So, there exists $k\ge m$ such that $x|_{T_k^e}\in \U$, which stands in contradiction with the definition of $D_{\U,m,\bm{\T}}$.
\end{proof}
\begin{lemma}\label{EumnG}
	Let $\U\supset \M_{\sigma}^{\mathsf{erg}}(X)$ be an open set in $\mathcal{B}^*(X)\cup\mathcal{M}(X)$ and let $m\in\N$. For every $n\ge m$ we define the set
	\begin{equation*}
	\bm{D}_{\U,m,n}=\{(x,\bm{\T})\in \bm{X}: \forall_{m\le k\le n}\hspace{3pt} x|_{T_k^e}\notin \U\}.
	\end{equation*}
	Furthermore, for every pair $(x,\bm{\T})\in \bm{X}$ we define
	\begin{equation*}
	\bm{E}_{\U,m,n,x,\bm{\T}}=\{(g,h)\in G\times G: (g(x),h(\bm{\T}))\in \bm{D}_{\U,m,n}\}.
	\end{equation*}
	Then the convergence $\lim\limits_{n\to+\infty}\overline{d}_{\mathsf{Ban}}(\bm{E}_{\U,m,n,x,\bm{\T}})=0$ holds uniformly on $\bm{X}$, where the upper Banach density is calculated in $G\times G$.
\end{lemma}
\begin{proof} Clearly, the sets $\bm{D}_{\U,m,n}$ form a nested sequence with respect to $n$. Hence, also, for each $(x,\bm{\T})\in \bm{X}$, the sets $\bm{E}_{\U,m,n,x,\bm{\T}}$ form a nested sequence. Thus, the sequence of numbers $(\overline{d}_{\mathsf{Ban}}(\bm{E}_{\U,m,n,x,\bm{\T}}))_{n\in\N}$ is nonincreasing. Observe that 
	\begin{equation*}
	\bm{D}_{\U,m,n}=\bigcap_{i=m}^n\bigcup_{S\in\mathcal{S}_{i}}\bigcup_{s\in S}(\left\{x\in X: x|_{Ss^{-1}}\notin \mathcal{U}\right\}\times\{\bm{\T}\in\bT: T^e_k=Ss^{-1} \})
	\end{equation*}
	is a clopen subset of $\bm{X}$. Therefore, the characteristic functions $\bm{1}_{\bm{D}_{\U,m,n}}$ are continuous on $\bm{X}$ and, consequently, for every $n\ge m$, the function ${\mu\mapsto\bm{\Phi}_{\U,m,n}(\mu)=\mu(\bm{D}_{\U,m,n})}$ is continuous on $\M_{(G\times G)}(\bm{X})$. Moreover, the descending intersection $\bigcap_{n\ge m}\overline{D}_{\U,m,n}=\overline{D}_{\U,m}$ is, by lemma~\ref{measure_zero}, a null set for every $(G\times G)$-invariant measure. Thereupon, by the continuity of measures from above, the sequence $(\bm{\Phi}_{\U,m,n})_{n\ge m}$ converges to the constant function equal to $0$, pointwise, on the compact set $\M_{(G\times G)}(\bm{X})$. Since the sequence $(\bm{\Phi}_{\U,m,n})_{n\ge m}$ is nonincreasing, by Dini's theorem, it converges to $0$ uniformly on $\M_{(G\times G)}(\bm{X})$. Thus, $$\lim\limits_{n\to+\infty} \sup\bigl\{\mu(\overline{D}_{\U,m,n}): \mu\in\mathcal{M}_{(G\times G)}(\bm{X})\bigr\}=0.$$ By lemma~\ref{upper_banach_density_invariantG}, this implies that $ \sup_{(x,\bm{\T})\in\bm{X}}\overline{d}_{\mathsf{Ban}}(\bm{E}_{\U,m,n,x,\bm{\T}})$ tends to $0$ as $n\to+\infty$, and consequently, the sequence of functions $(x,\bm{\T})\mapsto \overline{d}_{\mathsf{Ban}}(\bm{E}_{\U,m,n,x,\bm{\T}})$ converges to $0$ as $n\to+\infty$, uniformly on~$\bm{X}$.
\end{proof}

 In the proof of the main theorem of this section (i.e.\ theorem~\ref{thm1G}) we use also the following, technical lemma.
\begin{lemma}\label{Foelner_in_group}
	Let $K\subset G$ be a finite set and let $F\subset G$ be $(K,\tfrac{\varepsilon}2)$-invariant. Then the set
	\begin{equation}\label{LT1}
	L=\bigcup_{f\in F} \{(g,gf): g \in K\}
	\end{equation}
	is an $\varepsilon$-modification of the set $K\times F\subset G\times G$.
\end{lemma}
\begin{proof}
	Observe (see fig.~\ref{figtile}) that
	\begin{equation}\label{LT2}
	L=\bigcup_{g\in K} \{g\}\times gF.
	\end{equation}
	Since $F$ is $(K,\tfrac{\varepsilon}2)$-invariant, by fact~\ref{gandK} b), it is $(g,\varepsilon)$-invariant for all $g\in K$. Hence,
	\begin{equation*}
	|L\triangle (K\times F)|\le\Bigl|\medcup_{g\in K}\bigl((\{g\}\times gF)\triangle (\{g\}\times F)\bigr)\Bigr|\le\sum_{g\in K}|gF\triangle F| <|K||F|\varepsilon,
	\end{equation*}
	and consequently,
	\begin{equation*}
	\frac{|L\triangle (K\times F)|}{|K\times F|}<\varepsilon,
	\end{equation*}
	what was to be shown.
\end{proof}

 Now we will formulate and prove the generalization of theorem~\ref{thm1} to the case of symbolic systems with the action of a countable amenable group $G$. We continue to work in the setup introduced at the beginning of this section. Moreover, to abbreviate the notation, for a fixed $x\in X$ and a neighbourhood $\U$ of the set $\Me[X]$, we will say that a tile $Q=Sc$, where $S\in\mathcal{S}_k$, $k\in\N$, and $c\in G$, is $\U$-\emph{ergodic} if $x|_Q\in\U$. Tiles which are not $\U$-ergodic will be called shortly \emph{nonergodic}.
\begin{theorem}\label{thm1G}
	Let $\U\supset \Me[X]$ be an open set in $\mathcal{B}^*(X)\cup \mathcal{M}(X)$ and let $m\in\N$. Then, for every $\varepsilon>0$, there exists $n\ge m$ such that, for every $x\in X$, there exists a collection $\mathcal{Q}$ of pairwise disjoint, $\U$-ergodic tiles, of shapes belonging to $\bigcup_{k=m}^{n}\mathcal{S}_k$, such that $\bigcup_{Q\in \mathcal{Q}}Q$ has lower Banach density in $G$ at least $1-\varepsilon$.
\end{theorem}
\begin{proof}
Let $\varepsilon>0$ and $m\in\N$ be fixed. By lemma~\ref{EumnG}, there exists $n\ge m$ such that for all $(x,\bm{\T})\in\bm{X}$ we have
$\overline{d}_{\mathsf{Ban}}(\bm{E}_{\U,m,n,x,\bm{\T}})<\tfrac{\varepsilon}4$, where the upper Banach density is calculated in $G\times G$. We denote $K=\bigcup_{k=m}^{n}\bigcup_{S\in\mathcal{S}_k}S$ and choose $l_0$ such that, for all $l\ge l_0$, each shape $S\in\mathcal{S}_{l}$ is $(KK^{-1},\tfrac{\varepsilon}{2|K|^2})$-invariant. For every $l\ge l_0$ we put $\bm{\mathcal{S}}_l=\{S\times S': S,S'\in\mathcal{S}_l\}$. Note that the union $\bigcup_{l\ge l_0}\bm{\mathcal{S}}_l$, arranged in a sequence, is a F{\o}lner sequence in $G\times G$.
Thus, by theorem~\ref{density_foelner}, enlarging, if necessary, $l_0$, we can assume that for all sets $S\times S'\in \bm{\mathcal{S}}_{l_0}$ the following estimation is true
\begin{equation}\label{est}
\sup_{(g,h)\in G\times G}\frac{|\bm{E}_{\U,m,n,x,\bm{\T}}\cap (Sg\times S'h)|}{|S||S'|}<\frac{\varepsilon}4.
\end{equation}

 We fix a tile $T\in\T_{l_0}$. Let $S\in\mathcal{S}_{l_0}$ be the shape of $T$, and let $c$ be its center. By the equation~\eqref{est}, for every $h\in G$ we have
\begin{equation}\label{est2}
\frac{\varepsilon}4|T||S'|>| \bm{E}_{\U,m,n,x,\bm{\T}}\cap (Sc\times S'h)|=\bigl|\bm{E}_{\U,m,n,x,\bm{\T}}\cap (T\times S'h) \bigr|.
\end{equation}

 We now choose a shape $\hat{S}(T)\in\bigcup_{l\ge l_0+1}\mathcal{S}_l$, which is $(T,\tfrac{\varepsilon}8)$-invariant (we point out that, unless $G$ is abelian, $(S,\tfrac{\varepsilon}8)$-invariance is insufficient in the forthcoming argument). Since $\bT$ is a deterministic system of tilings, $\hat{S}(T)$ is a union of disjoint shapes belonging to $\mathcal{S}_{l_0}$: $\hat{S}(T)=\bigcup_{j=1}^pS_{j}c_j$, where $S_{j}\in\mathcal{S}_{l_0}$ and $c_j$ are some elements of $G$. Thus, from the equation~\eqref{est2} it follows that
\begin{equation*}
\frac{\bigl|\bm{E}_{\U,m,n,x,\bm{\T}}\cap (T\times \hat{S}(T)) \bigr|}{|\hat{S}(T)||T|}=\sum_{j=1}^p\frac{|S_{j}|}{|\hat{S}(T)|}\frac{\bigl|\bm{E}_{\U,m,n,x,\bm{\T}}\cap (T\times S_{j}c_j) \bigr|}{|S_{j}||T|}<\frac{\varepsilon}{4}.
\end{equation*}
	\begin{figure}[h]
	\centering
	\begin{tikzpicture}
	\begin{axis}[
	axis lines = center,
	xlabel = $G$,
	ylabel = $G$,
	xmin=-28,
	xmax=28,
	ymin=-28,
	ymax=28,
	x=0.165cm, y=0.165cm,
	ticks=none,
	every axis x label/.style={
		at={(ticklabel* cs:1)},
		anchor=west,
	}
	]
	\draw (axis cs: -1.5,-6) rectangle (axis cs: 1.5,6);
	\draw (axis cs: 1.5,-14) rectangle (axis cs: 4.5,14);
	\draw (axis cs: 4.5,-20) rectangle (axis cs: 7.5,20);
	\draw (axis cs: 7.5,-30) rectangle (axis cs: 10.5,30);
	\draw (axis cs: -4.5,-14) rectangle (axis cs: -1.5,14);
	\draw (axis cs: -7.5,-20) rectangle (axis cs: -4.5,20);
	\draw (axis cs: -10.5,-30) rectangle (axis cs: -7.5,30);
	\draw [name path=T01] (axis cs: -1.5,-7.5) -- (axis cs: -1.5,4.5)--(axis cs: 1.5,7.5);
	\draw [name path=T02] (axis cs: -1.5,-7.5)--(axis cs: 1.5,-4.5)--(axis cs: 1.5,7.5);
	\fill [pattern=north east lines,
	intersection segments={
		of=T01 and T02,
		sequence={L2--R2}
	}];
	\draw [name path=T11] (axis cs: 1.5,-11.5) -- (axis cs: 1.5,15.5)--(axis cs: 4.5,18.5);
	\draw [name path=T12] (axis cs: 1.5,-11.5)--(axis cs: 4.5,-9.5)--(axis cs: 4.5,18.5);
	\fill [pattern=vertical lines,
	intersection segments={
		of=T11 and T12,
		sequence={L2--R2}
	}];
	\draw [name path=T21] (axis cs: 4.5,-15.5) -- (axis cs: 4.5,24.5)--(axis cs: 7.5,27.5);
	\draw [name path=T22] (axis cs: 4.5,-15.5)--(axis cs: 7.5,-12.5)--(axis cs: 7.5,27.5);
	\fill [pattern=north east lines,
	intersection segments={
		of=T21 and T22,
		sequence={L2--R2}
	}];
	\draw [name path=T31] (axis cs: 7.5,-22.5) -- (axis cs: 7.5,37.5)--(axis cs: 10.5,40.5);
	\draw [name path=T32] (axis cs: 7.5,-22.5)--(axis cs: 10.5,-19.5)--(axis cs: 10.5,40.5);
	\fill [pattern=vertical lines,
	intersection segments={
		of=T31 and T32,
		sequence={L2--R2}
	}];
	\draw [name path=T-11] (axis cs: -1.5,12.5)--(axis cs: -4.5,9.5)--(axis cs: -4.5,-18.5);
	\draw [name path=T-12] (axis cs: -1.5,12.5)--(axis cs: -1.5,-15.5)--(axis cs: -4.5,-18.5);
	\fill [pattern=vertical lines,
	intersection segments={
		of=T-11 and T-12,
		sequence={L2--R2}
	}];
	\draw [name path=T-21] (axis cs: -4.5,15.5)--(axis cs: -7.5,12.5)--(axis cs: -7.5,-27.5);
	\draw [name path=T-22] (axis cs: -4.5,15.5)--(axis cs: -4.5,-24.5)--(axis cs: -7.5,-27.5);
	\fill [pattern=north east lines,
	intersection segments={
		of=T-21 and T-22,
		sequence={L2--R2}
	}];
	\draw [name path=T-31] (axis cs: -7.5,22.5)--(axis cs: -10.5,19.5)--(axis cs: -10.5,-40.5);
	\draw [name path=T-32] (axis cs: -7.5,22.5)--(axis cs: -7.5,-37.5)--(axis cs: -10.5,-40.5);
	\fill [pattern=vertical lines,
	intersection segments={
		of=T-31 and T-32,
		sequence={L2--R2}
	}];
	\addplot [
	domain=-12:12, 
	samples=100, 
	color=black,
	dashed,
	line width=1.2pt,
	]
	{x};
	\draw[very thick] (axis cs: 8,8)--(axis cs: 9.5,6)--(axis cs: 12,6) node[pos=1,anchor=west]{$\{(g,g): g\in G\}$};
	\draw[line width=3pt] (axis cs: 4.5,0) -- (axis cs: 7.5,0) node[pos=0.5, anchor=south, ]{$T$};
	\draw[line width=3pt] (axis cs: 4.5,-20) -- (axis cs: 4.5,20) node[pos=0.07, anchor=south, rotate=90]{$\hat{S}(T)$};
	\draw[very thick] (axis cs: 6,22) -- (axis cs: 12,26)-- (axis cs: 15, 26) node[pos=1, anchor=west]{$L(T)$};
	\draw[very thick] (axis cs: 6,-17) -- (axis cs: 12,-22)-- (axis cs: 15, -22) node[pos=1, anchor=west]{$T\times\hat{S}(T)$};
	\end{axis}
	\end{tikzpicture}
	\caption{The scheme of choosing the sets $\hat{S}(T)$ and construction of the sets $L(T)$ (diagonal hatching corresponds to the equation~\eqref{LT1}, whereas vertical hatching corresponds to the equation~\eqref{LT2})}\label{figtile}
	\label{GtimesG}
\end{figure}
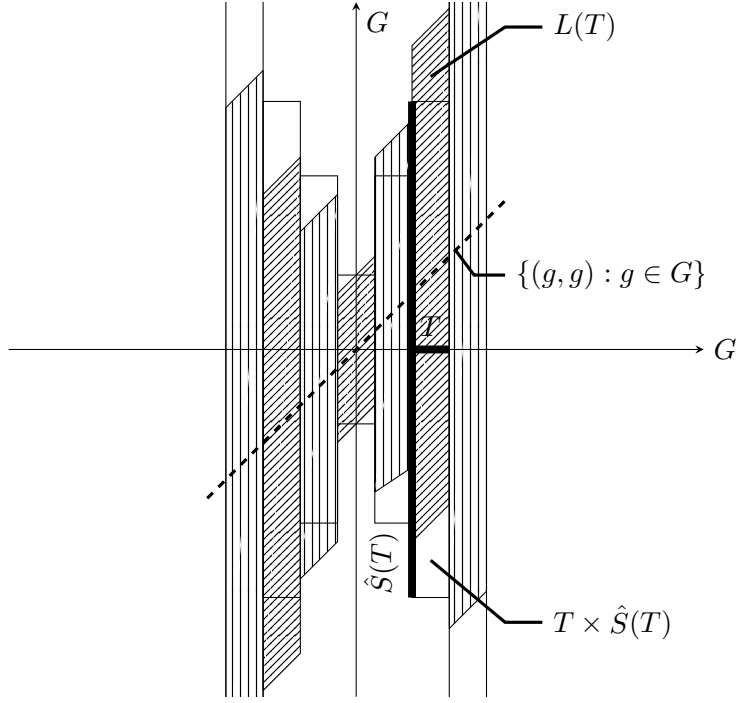

 By $(T,\tfrac{\varepsilon}8)$-invariance of $\hat{S}(T)$ and by lemma~\ref{Foelner_in_group}, the set $$L(T)=\bigcup_{h\in \hat{S}(T)}\{(g,gh):g\in T\}$$ is an $\tfrac{\varepsilon}4$-modification of $T\times \hat{S}(T)$. Moreover, it has the same cardinality as $T\times \hat{S}(T)$. Therefore, the following inequality holds
\begin{equation}\label{est3}
\frac{\bigl|\bm{E}_{\U,m,n,x,\bm{\T}}\cap L(T) \bigr|}{|L_{T}|}\le \frac{|\bm{E}_{\U,m,n,x,\bm{\T}}\cap (T\times \hat{S}(T))|}{|T\times \hat{S}(T)|}+\frac{| L(T)\setminus (T\times \hat{S}(T))|}{|T\times \hat{S}(T)|}<\frac{\varepsilon}2.
\end{equation}
Because $L(T)$ is a disjoint union (over $h\in\hat{S}(T)$) of the sets $\{(g,gh):g\in T\}$, each of cardinality $|T|$, there exists at least one element $h_{T}\in \hat{S}(T)$ such that
\begin{equation*}
\frac{|\bm{E}_{\U,m,n,x,\bm{\T}}\cap\{(g,gh_{T}): g\in T\}|}{|T|}<\frac{\varepsilon}2.
\end{equation*}
Observe that $$\bm{E}_{\U,m,n,x,\bm{\T}}\cap\{(g,gh_T): g\in T\}=\bm{E}_{\U,m,n,x,h_{T}(\bm{\T})}\cap \{(g,g): g\in T\}.$$
Denoting the set $\{g\in G: (g,g)\in \bm{E}_{\U,m,n,x,h_{T}(\bm{\T})}\}$ by $E_{\U,m,n,x,h_{T}(\bm{\T})}$, we also have
$$|\bm{E}_{\U,m,n,x,h_{T}(\bm{\T})}\cap \{(g,g): g\in T\}|=|E_{\U,m,n,x,h_{T}(\bm{\T})}\cap T|.$$
Henceforth, we have shown the inequality
\begin{equation}\label{ineq}
\frac{|E_{\U,m,n,x,h_T(\bm{\T})}\cap T|}{|T|}<\frac{\varepsilon}2.
\end{equation}

 Because every shape $S\in\mathcal{S}_{l_0}$, as well as every tile $T\in\T_{l_0}$, is $(KK^{-1},\tfrac{\varepsilon}{2|K|^2})$-invariant, by lemma~\ref{lemma-core}, for every $T\in\T_{l_0}$ we have
\begin{equation*}
\frac{|T\setminus T_{KK^{-1}}|}{|T|}< \frac{\varepsilon}2.
\end{equation*}
The core $T_{KK^{-1}}$ has the property, that for every shape $S\in\bigcup_{k=m}^n\mathcal{S}_k$ and $c\in G$ the following implication holds 
\begin{equation}\label{implication_core}
Sc\cap T_{KK^{-1}}\neq \varnothing \Rightarrow Sc\subset T.
\end{equation} 

 Within $T$ we select a family $\mathcal{Q}(T)$ of $\U$-ergodic tiles, as follows. By the definition of the set $E_{\U,m,n,x,h_{T}(\bm{\T})}$, for every $g\notin E_{\U,m,n,x,h_{T}(\bm{\T})}$, there exists $k\in[m,n]$ for which the tile $T_{k}^g= T_k^g\bigl(h_{T}(\T_k)\bigr) $ satisfies $x|_{T_k^g}\in \U$, i.e.\ $T_k^g$ is $\U$-ergodic. For every $g\in T_{KK^{-1}}\!\!\setminus\!E_{\U,m,n,x,h_{T}(\bm{\T})}$ let $k(g)$ denote the largest such $k\in[m,n]$. Since $\bm{\T}$ belongs to a deterministic system of tilings, for $g\neq g'\in T_{KK^{-1}}\!\!\setminus\! E_{\U,m,n,x,h_{T}(\bm{\T})}$, the tiles $T_{k(g)}^g,\ T_{k(g')}^{g'}$ are either disjoint, or one of them is included in the other. However, the way the tiles $T_{k(g)}^g$ and $T_{k(g')}^{g'}$ were chosen excludes the possibility of strict inclusion. Thence, every two of the chosen tiles are either disjoint or equal. Note also that, by~\eqref{implication_core}, all the tiles $T_{k(g)}^g$, $g\in T_{KK^{-1}}\!\setminus E_{\U,m,n,x,h_{T}(\bm{\T})}$, are contained in $T$. We denote the collection of tiles $T_{k(g)}^g$, constructed this way, by $\mathcal{Q}(T)$.  We repeat the above construction for all $T\in\T_{l_0}$. Then we put $\mathcal{Q}=\bigcup_{T\in\T_{l_0}}\mathcal{Q}(T)$. All the tiles $Q\in\mathcal{Q}$ are $\U$-ergodic. It is worth to mention, that by the equation~\eqref{implication_core}, for every $T\in \T_{l_0}$, the following inclusion holds
\begin{equation*}
T_{KK^{-1}}\!\!\setminus\! E_{\U,m,n,x,h_{T}(\bm{\T})} \subset\bigcup_{Q\in\mathcal{Q}(T)}Q\subset T.
\end{equation*}
On account of that, we have
\begin{multline}\label{nie_wiem}
\frac{\bigl|T\cap \medcup_{Q\in\mathcal{Q}}Q\bigr|}{|T|}=\frac{\bigl| \medcup_{Q\in\mathcal{Q}(T)}Q\bigr|}{|T|}\ge \frac{|T_{KK^{-1}}\!\setminus\! E_{\U,m,n,x,h_{T}(\bm{\T})}|}{|T|}\\\ge \frac{|T\setminus\! E_{\U,m,n,x,h_{T}(\bm{\T})}|}{|T|}-\frac{|T\setminus\! T_{KK^{-1}}|}{|T|}>1-\frac{\varepsilon}{2}-\frac{\varepsilon}{2}=1-\varepsilon.
\end{multline}
From which, by lemma~\ref{density_tiling}, it follows that
\begin{equation*}
\underline{d}_{\mathsf{Ban}}\Bigl(\bigcup_{Q\in\mathcal{Q}}Q\Bigr)\ge 1-\varepsilon.
\end{equation*}
\end{proof}
\begin{remark}
	In case $\U$ contains a ball $\mathsf{Ball}(\M_{\sigma}^{\mathsf{erg}}(X),\rho)$, $\rho>0$, it is possible to find a tiling $\mathcal{Q}'$ of $G$, consisting exclusively of $\U$-ergodic tiles. The construction of the tiling $\mathcal{Q}'$ relies on modifying the collection $\mathcal{Q}$ obtained in theorem~\ref{thm1G} for $\tfrac{\varepsilon}2<\tfrac{\rho}2$ and a neighbourhood $\V=\mathsf{Ball}(\Me[X],\tfrac{\rho}2)$ (in place of $\U$), by appropriately distributing the elements of the complement of $\bigcup_{Q\in \mathcal{Q}}Q$ amongst the tiles ${Q\in\mathcal{Q}}$. The construction follows the same path (based on a variant of Hall's marriage theorem) as the proof of \cite[theorem 4.3]{DHZ}. As a result, the shapes of the tiles $Q'\in \mathcal{Q}'$ are $\tfrac{\varepsilon}2$-modifications of the shapes belonging to $\bigcup_{k=m}^n\mathcal{S}_k$. However, one has to bear in mind, that in case $G=\Z$, the tiles of $Q'$ will typically not be intervals (they will have the form of a union of an interval and a small amount of isolated points). As our examples~\ref{nontoepliz} and~\ref{toeplitz} show, in some cases, a tiling $\mathcal{Q}'$ whose all tiles are $\U$-ergodic intervals does not exist.
\end{remark}

 We end this section with a formulation and proof of a generalization of theorem~\ref{thm2} to the case of $G$-subshifts. In the proof of theorem~\ref{thm2G} we use the following straightforward generalization of lemma~\ref{konkatenacja} to the case of $G$-subshifts.
\begin{lemma}\label{konkatenacjaG}
	Let $G$ be a countable amenable group and let $X$ be a symbolic system with the action of $G$. For every $\varepsilon>0$ and any collection of finite blocks $B_1,...,B_m\in \mathcal{B}^*(X)$ with pairwise disjoint domains $F_1,...,F_m\subset G$, such that for every $j=1,...,m$, there exists a measure $\mu_j\in\mathcal{M}(X)$ satisfying $d^*(B_j,\mu_j)<\varepsilon$, the following inequality holds
	\begin{equation*}
	d^*\left(B,\mu\right)<2\varepsilon,
	\end{equation*}
	where $B$ is the concatenation of the blocks $B_1,..., B_m$, that is, the block with the domain $F=\bigcup_{j=1}^mF_j$, such that $B|_{F_j}= B_j$ for $j=1,...,m$, and $\mu=\sum_{j=1}^{m}\frac{|F_j|}{|F|}\mu_j$.
\end{lemma}
\begin{theorem}\label{thm2G}
		Let $(X,\sigma)$ be a symbolic system with the action of a countable amenable group $G$, such that $\mathcal{M}_{\sigma}(X)$ is a Bauer simplex. Let $\bT$ be a F{\o}lner, deterministic system of dynamical tilings of $G$. Let $\U\supset \Me[X]$ be an open subset of $\mathcal{B}^*(X)\cup\mathcal{M}(X)$ and fix an $\varepsilon>0$. Then, there exists $j_0\in\mathbb{N}$ such that for every $j\ge j_0$ and every pair $(x,\bm{\T})\in \bm{X}=X\times \bT$, the union $M^{\mathsf{NE}}(x,\T_j)$ of the nonergodic tiles of $\T_{j}$ has upper Banach density in $G$ smaller than~$\varepsilon$.
\end{theorem}
\begin{proof}
	
	 Since $\Me[X]$ is a compact set, without loss of generality, we can assume that $\U=\mathsf{Ball}(\Me[X],\varepsilon)$. By lemma~\ref{barycenter_compact}, there exists $0<\gamma<\tfrac{4\varepsilon}{3}$ such that for every $\mu_0\in\M_{\sigma}^{\mathsf{erg}}(X)$ and every $\mu=\int_{\mathcal{M}_{\sigma}(X)}\nu\mathrm{d}\xi(\nu)$, the following implication holds
	\begin{equation*}
	d^*(\mu,\mu_0)<\gamma \Rightarrow \xi\bigl(\mathcal{M}_{\sigma}(X)\setminus \mathsf{Ball}(\mu_0,\tfrac{\varepsilon}3)\bigr)<\tfrac{\varepsilon}3.
	\end{equation*}
	
	 Let $j_0$ be such that for all blocks $B\in\mathcal{B}^*(X)$ with domains being shapes $S\in\bigcup_{j\ge j_0}\mathcal{S}_j$ we have $d^*(B,\mathcal{M}_{\sigma}(X))<\tfrac{\gamma}4$. Existence of such $j_0$ follows from theorem~\ref{blocks_close_to_measuresG}. We fix a pair $(x,\bm{\T})\in\bm{X}$, $j\ge j_0$ and we let $\T=\T_{j}$. We put $K=\bigcup_{S\in\mathcal{S}_{j}}S$. Let $m\in\N$ be such that for all $k\ge m$, every $S\in\mathcal{S}_k$ is $(KK^{-1},\tfrac{\gamma}{4})$-invariant. We put $\V=\mathsf{Ball}(\Me[X],\tfrac{\gamma}4)$. By theorem~\ref{thm1}, there exist $n\ge m$ and a collection $\mathcal{Q}$ of pairwise disjoint, $\V$-ergodic tiles with shapes belonging to $\bigcup_{k=m}^{n}\mathcal{S}_k$, such that the union of tiles $\bigcup_{Q\in \mathcal{Q}}Q$ has lower Banach density at least~$1-\tfrac{\varepsilon}3$. In what follows, tiles $Q\in \mathcal{Q}$ will be called ``auxiliary ergodic''. For every ``auxiliary ergodic'' tile $Q\in\mathcal{Q}$, there exists a measure $\mu_{Q}\in\M_{\sigma}^{\mathsf{erg}}(X)$ such that $d^*(x|_{Q},\mu_{Q})<\tfrac{\gamma}4$.
	
	 We fix an ``auxiliary ergodic'' tile $Q\in\mathcal{Q}$. Since every shape $S\in\bigcup_{k=m}^{n}\mathcal{S}_k$ is $(KK^{-1},\tfrac{\gamma}4)$-invariant, every $Q\in\mathcal{Q}$ is $(KK^{-1},\tfrac{\gamma}4)$-invariant too. Hence, by lemma~\ref{lemma-core}, we have
	\begin{equation*}
	\frac{|Q\setminus Q_{KK^{-1}}|}{|Q|}<\frac{\gamma}{4}.
	\end{equation*}
	Let $\T(Q)$ denote the collection of those tiles $T\in\T$ which are not disjoint with the core $Q_{KK^{-1}}$ and let $Q'=(Q_{KK^{-1}})^{\T}=\bigcup_{T\in \T(Q)}T$. Observe that $Q_{KK^{-1}}\subset Q'\subset Q$.
	Thus, the following inequality holds
	\begin{equation}\label{12}
	\frac{|Q\setminus Q'|}{|Q|}\le \frac{\bigl|Q\setminus Q_{KK^{-1}}\bigr|}{|Q|}<\frac{\gamma}{4}.
	\end{equation}
	Henceforth $Q'$ is $\tfrac{\gamma}4$-modification of $Q$. So, by lemma~\ref{close_mod}, it is true that
	\begin{equation*}
	d^*(x|_{Q'},\mu_Q)\le d^*(x|_{Q'},x|_{Q})+d^*(x|_{Q},\mu_Q)<\frac{\gamma}{4}+\frac{\gamma}{4}=\frac{\gamma}{2}.
	\end{equation*}
	
	 Recall that for every $T\in\T$ we have $d^*(x|_T,\mathcal{M}_{\sigma}(X))<\tfrac{\gamma}4$. So, for every $T\in\T(Q)$, there exists an invariant measure $\nu_T\in \mathcal{M}_{\sigma}(X)$ such that $d^*(x|_T,\nu_T)<\tfrac{\gamma}4$. From lemma~\ref{konkatenacjaG} it follows that
	\begin{equation*}
	d^*\Bigl(x|_{Q'},\sum_{T\in\T(Q)}\frac{|T|}{|Q'|}\nu_T\Bigr)<\frac{\gamma}2.
	\end{equation*}
	Using the triangle inequality we obtain
	\begin{equation*}
	d^*\Bigl(\mu_{Q},\sum_{T\in\T(Q)}\frac{|T|}{|Q'|}\nu_T\Bigr)\le d^*(\mu_{Q},x|_{Q'})+d^*\Bigl(x|_{Q'},\sum_{T\in\T(Q)}\frac{|T|}{|Q'|}\nu_T\Bigr)<\gamma.
	\end{equation*}
	Clearly, $\sum_{T\in\T(Q)}\frac{|T|}{|Q'|}\nu_T=\int_{\mathcal{M}_{\sigma}(X)}\nu\mathrm{d}\xi(\nu)$ for the measure $\xi=~\sum_{T\in\T(Q)}\frac{|T|}{|Q'|}\delta_{\nu_T}$, where $\delta_{\nu_T}$ is the Dirac measure on the Bauer simplex $\mathcal{M}_{\sigma}(X)$, supported at $\nu_T$. By theorem~\ref{barycenter_compact}, the sum of coefficients $\frac{|T|}{|Q'|}$ corresponding to $T$ belonging to the set
	$$	\{T\in\mathcal{T}(Q): \nu_T\in \mathcal{M}_{\sigma}(X)\setminus \mathsf{Ball}(\mu_{Q},\tfrac{\varepsilon}3)\}$$
	is smaller than $\tfrac{\varepsilon}3$.
	
	 We denote by $\mathcal{T}^{\mathsf{NE}}(Q)$ the set $\{T\in\mathcal{T}(Q): d^*(x|_T,\M_{\sigma}^{\mathsf{erg}}(X))\ge \varepsilon \}$ (i.e.\ the collection of nonergodic tiles $T$ included in the ``auxiliary ergodic'' tile $Q$).
	Observe that for $T\in \T^{\mathsf{NE}}(Q)$, by the triangle inequality, we have
	\begin{equation*}
	d^*(\nu_T,\mu_{Q})\ge d^*(x|_{T},\mu_{Q})-d^*(x|_{T},\nu_T)\ge \varepsilon-\frac{\gamma}4>\frac{\varepsilon}3.
	\end{equation*} 
	Thence, the following inclusion holds
	\begin{equation*}
	\mathcal{T}^{\mathsf{NE}}(Q)\subset \{T\in\mathcal{T}(Q): \nu_T\in \mathcal{M}_{\sigma}(X)\setminus \mathsf{Ball}(\mu_{Q},\tfrac{\varepsilon}3)\}.
	\end{equation*}
	Therefore $\sum_{T\in\T^{\mathsf{NE}}(Q)}\frac{|T|}{|Q'|}<\tfrac{\varepsilon}3$. In other words, the fraction of nonergodic tiles $T$ in the fixed ``auxiliary ergodic'' tile $Q$ is smaller than~$\tfrac{\varepsilon}3$.
	
	 Let $M^{\mathsf{NE}}(x,\T)$ denote the union of all nonergodic tiles $T\in\T$ and $M^{\mathsf{NE}}(Q)=\ \bigcup_{T\in \T^{\mathsf{NE}}(Q)}T$ denote the union of all nonergodic tiles $T$ included in the fixed ``auxiliary ergodic'' tile $Q$. Then we have
	\begin{equation}\label{epsilon3}
	\frac{|M^{\mathsf{NE}}(Q)|}{|Q|}=\frac{|Q'|}{|Q|}\sum_{T\in\T^{\mathsf{NE}}(Q)}\frac{|T|}{|Q'|}<\frac{\varepsilon}{3}.
	\end{equation}
	
	 Recall that our goal is to show that $\overline{d}_{\mathsf{Ban}}(M^{\mathsf{NE}}(x,\T))<\varepsilon$. In the construction of the collection $\mathcal{Q}$ (see the proof of theorem~\ref{thm1G}) we have used the tiling $\T_{l}$, from now on denoted by $\mathcal{P}$, with the property that every tile $P\in\mathcal{P}$ satisfies
	\begin{equation}\label{epsilon6}
	\frac{|P\cap\medcup_{Q\in\mathcal{Q}}Q|}{|P|}>1-\frac{\varepsilon}3
	\end{equation}
	(see the equation~\eqref{nie_wiem} -- we remind that the collection of tiles $\mathcal{Q}$ occurring in this proof is constructed using theorem~\ref{thm1G} with $\tfrac{\varepsilon}3$ in place of $\varepsilon$).
	On the account of lemma~\ref{density_tiling} and remark~\ref{densities_comp}, it suffices to show that for every $P\in\mathcal{P}$ the following inequality holds
	\begin{equation*}
	\frac{|P\cap M^{\mathsf{NE}}(x,\T)|}{|P|}<\varepsilon.
	\end{equation*}
	By the equations~\eqref{epsilon6},~\eqref{12} and~\eqref{epsilon3} we obtain
	\begin{multline*}
	\frac{|P\cap M^{\mathsf{NE}}(x,\T)|}{|P|}\le \frac{|P\setminus\medcup_{Q\in\mathcal{Q}}Q|}{|P|}+\sum_{Q\subset P}\frac{|Q|}{|P|}\frac{|Q\setminus Q'|}{|Q|}+\sum_{Q\subset P}\frac{|Q|}{|P|}\frac{|M^{\mathsf{NE}}(Q)|}{|Q|}\\<\frac{\varepsilon}{3}+\frac{\gamma}{4}+\frac{\varepsilon}{3}<\varepsilon,
	\end{multline*}
	which completes the proof.
\end{proof}
\section{Final remarks}
Firsty, we would like to remark that the results presented in the previous section can be directly transferred to the case of countable amenable cancellative semigroups, since every such semigroup can be naturally embedded in a countable amenable group in such a way that a fixed F{\o}lner sequence in the semigroup becomes a F{\o}lner sequence in the group.

Secondly, we point out that the main theorems of this paper are valid (after an appropriate reformulation, see below) not only for subshifts but also for all classical topological dynamical systems $(X,T)$ (with an action of $\Z$ or $\N_0$ on a compact metric space $X$) and for general topological dynamical systems $(X,\tau)$ (with actions of a countable amenable group $G$ on a compact metric space $X$). Instead of blocks occurring in $x\in X$, say $B=x|_{[i,i+k)}$, one has to consider ``pieces of orbits'' of the form $\{T^j(x): j\in[i,i+k)\}$ (resp.\ $\{g(x): g\in K\}$ instead of $B=x|_K$ for $K\subset G$). Then, instead of the empirical measure associated with $B$ one has to consider simply the probability measure $\frac1k\sum_{j=0}^{k-1}\delta_{T^{i+j}(x)}$ (resp. $\frac1{|K|}\sum_{g\in K}\delta_{g(x)}$). Most of the proofs actually simplify, for example, it suffices to consider the metric $d^*$ on $\M(X)$ without needing to extend it to $\mathcal B^*(X)$, also, lemma~\ref{konkatenacja} (resp.~\ref{konkatenacjaG}), is not needed. However, the simplification causes that there is no direct way of deducing theorems for symbolic systems from their general analogs; for instance, there are subtle differences between the metric $d^*$ on $\M(X)$ and the extended pseudometric on $\mathcal B^*(X)\cup\M(X)$. This is one of the reasons why we have chosen to write all the proofs for symbolic systems rather than the easier proofs for general topological systems. For completeness, let us formulate the main theorems in the general setup of countable amenable group actions:

\begin{theorem}
Let $\tau$ be an action of a countable amenable group $G$ on a compact metric space $X$. Let $\U$ be an open set in $\M(X)$ containing all ergodic measures of the action $\tau$. Let $\bT=\bigvee_{k\in\N}\mathsf T_k$ be a F\o lner, deterministic system of tilings of $G$ and let $\mathcal S_k$ denote the collection of shapes of $\mathsf T_k$, $k\in\N$. Then, for every $\varepsilon>0$ there exists $n\ge m$ such that for every $x\in X$ there is a collection $\mathcal Q$ of pairwise disjoint tiles with shapes belonging to $\bigcup_{k=m}^n\mathcal S_k$, whose union has lower Banach density at least $1-\varepsilon$ and for every $Q\in\mathcal Q$ we have $\frac1{|Q|}\sum_{g\in Q}\delta_{g(x)}\in\U$.
\end{theorem}

\begin{theorem}
Let $\tau$ be an action of a countable amenable group $G$ on a compact metric space $X$, such that the set of $\tau$-\im s, $\M_\tau(X)$, is a Bauer simplex. Let $\U$ be an open set in $\M(X)$ containing all ergodic measures of the action $\tau$. Let $\bT=\bigvee_{k\in\N}\mathsf T_k$ be a F\o lner, deterministic system of tilings of $G$. Then, for every $\varepsilon>0$ there exists $j_0\in\N$ such that for every $j\ge j_0$ and every pair $(x,\bm{\T})$, $x\in X$, $\bm{\T}=(\T_k)_{k\in\N}\in\bT$, the union $M^{\mathsf{NE}}(x,\T_j)$ of tiles $T$ of $\T_j$ such that $\frac1{|T|}\sum_{g\in T}\delta_{g(x)}\notin\U$ has upper Banach density smaller than $\varepsilon$.
\end{theorem}

\end{document}